\documentclass[11pt]{article} %was fullpage
\usepackage{amsfonts}
\usepackage{amsmath}
\usepackage{amssymb}
\usepackage{tikz}
\usepackage{mathabx}
\usepackage{authblk}
\usepackage[margin=1in]{geometry}
\usepackage{amssymb}
\usepackage{amsthm}
\usepackage{dsfont}
\usepackage{graphicx}
\usepackage{enumerate}
\usepackage[hypcap]{caption}
\usepackage{color, soul}
\usepackage{float}
\usepackage[ruled]{algorithm2e}
\usepackage{prettyref}
\usepackage{enumitem}
\usepackage{comment}

\usepackage{mathrsfs}
\usepackage{subcaption}

\usepackage{needspace}

\newcommand{\OPT}{{\rm OPT}(G)}

\usepackage{amsmath,amssymb,color}

\newcommand{\vol}{{\rm vol}}

\newcommand{\eps}{\varepsilon}

\renewcommand{\AA}{ \mathcal{A} }

\newcommand{\cA}{\mathcal{A}}
\newcommand{\cB}{\mathcal{B}}
\newcommand{\CC}{\mathcal{C}}
\newcommand{\cC}{\mathcal{C}}
\newcommand{\cQ}{\mathcal{Q}}
\newcommand{\q}{q^*}

\newcommand{\hdelta}{\hat{\delta}}
\newcommand{\hh}{\hat{h}}
\newcommand{\underscore}{\texttt{\detokenize{_}}}
\newcommand{\maxin}{\max\!\underscore{\rm in}}
\newcommand{\maxout}{\max\!\underscore{\rm out}}

\newtheorem{theorem}{Theorem}[section]

\newtheorem{corollary}[theorem]{Corollary}

\newtheorem{lemma}[theorem]{Lemma}
\newtheorem{claim}[theorem]{Claim}
\newtheorem{proposition}[theorem]{Proposition}

\newtheorem{remark}[theorem]{Remark}

\newtheorem{open}[theorem]{Open Question}

\numberwithin{figure}{section}
\numberwithin{equation}{section}

\usepackage{rotating}

\newcommand{\NN}{{\mathbb N}}

\renewcommand{\eps}{\varepsilon}

\usepackage{algorithm2e}
\usepackage{xcolor}

\newcommand{\xvar}[1]{\textsf{#1}}

\newcommand{\xvbox}[2]{\makebox[#1][l]{#2}}

\SetAlCapSty{xAlCapSty}

\SetCommentSty{xCommentSty}

\SetNlSty{mynlfont}{}{} 

\LinesNumbered

\SetSideCommentRight

\DontPrintSemicolon

\RestyleAlgo{algoruled}

\title{Modularity and Graph Expansion}

\begin{document}

\author[1]{Baptiste Louf\thanks{Email: \textit{baptiste.louf@math.u-bordeaux.fr}. Part of this work was done while supported by the Knut and Alice Wallenberg Foundation.}}%{Department of Mathematics, Uppsala University, Sweden}{fiona.skerman@math.uu.se}{https://orcid.org/0000-0003-4141-7059}

\author[2]{Colin McDiarmid \thanks{Email: \textit{cmcd@ox.ac.uk}}}%{Department of Mathematics, Uppsala University, Sweden}{fiona.skerman@math.uu.se}{https://orcid.org/0000-0003-4141-7059}

\author[3]{Fiona Skerman\thanks{Email: \textit{fiona.skerman@math.uu.se}. Partially supported by the Wallenberg AI, Autonomous Systems and Software Program WASP and the project AI4Research at Uppsala University. Part of this work was done while visiting the Simons Institute for the Theory of Computing, supported by a Simons-Berkeley Research Fellowship.}}%{Department of Mathematics, Uppsala University, Sweden}{fiona.skerman@math.uu.se}{https://orcid.org/0000-0003-4141-7059}

\affil[1]{CNRS and Institut de Math\'ematiques de Bordeaux, France}
\affil[2]{Department of Statistics, University of Oxford, United Kingdom}
\affil[3]{Department of Mathematics, Uppsala University, Sweden}

\maketitle

%\author{Baptiste Louf}{CNRS and Institut de Math\'ematiques de Bordeaux, France}{baptiste.louf@math.u-bordeaux.fr}{}{Part of this work done while supported by the Knut and Alice Wallenberg Foundation.}

%\author{Colin McDiarmid}{Department of Statistics, University of Oxford, United Kingdom}{cmcd@ox.ac.uk}{}{ }

%{Department of Mathematics, Uppsala University, Sweden}{fiona.skerman@math.uu.se}{https://orcid.org/0000-0003-4141-7059}{Partially supported by the Wallenberg AI, Autonomous Systems and Software Program WASP and Uppsala University AI4Research. Part of this work done while visiting the Simons Institute for the Theory of Computing, supported by a Simons-Berkeley Research Fellowship.}

%\authorrunning{B. Louf, C. McDiarmid, F. Skerman}

%#10003633
%\ccsdesc[500]{Mathematics of computing~Graph theory}
%\ccsdesc[500]{Mathematics of computing~Extremal graph theory}
%\ccsdesc[500]{Networks~Network performance evaluation}
%\ccsdesc[500]{Theory of computation~Theory and algorithms for application domains}
%\ccsdesc[500]{Mathematics of computing~Spectra of graphs}
%\ccsdesc[10003633]{\textcolor{red}{Replace ccsdesc macro with valid one}}

%\Copyright{ }

%\keywords{edge expansion, modularity, community detection, resolution limit, conductance}

%\relatedversion{A full version of the paper is available at \url{www.arxiv}~\cite{modexparxiv}.}%todo ITCS

%\acknowledgements{We are grateful to Prasad Tetali for helpful discussions.} 

\begin{abstract}
We relate two important notions in  graph theory: expanders which are highly connected graphs, and modularity a parameter of a graph that is primarily used in community detection. More precisely, we show that a graph having modularity bounded below 1 is equivalent to it having a large subgraph which is an expander. \\

We further show that a connected component $H$ will be split in an optimal partition of the host graph $G$ if and only if the relative size of $H$ in $G$ is greater than an expansion constant of $H$.
This is a further exploration of the resolution limit known for modularity, and indeed recovers the bound that a connected component $H$ in the host graph~$G$ will not be split if~$e(H)<\sqrt{2e(G)}$.
\end{abstract}

%%%%%%%%%%%%%%%%%%%%%%%%%%%%%%%%%%
%%%%%%%%%%%%%%%%%%%%%%%%%%%%%%%%%%

\section{Introduction and results}
\label{sec.intro}

In this paper, we relate the property of a graph having expander  induced subgraphs with the notion of its modularity.
(We recall the definitions of the modularity of a graph and of expander graphs in Section~\ref{subsec.defns} below.)  Modularity was introduced by Newman and Girvan~\cite{NewmanGirvan}, and it quantifies how well a graph can be divided into separate communities. More precisely, to each graph $G$, we associate a quantity $0 \leq \q(G) <1$, and the higher~$q^*(G)$ is the more `community structure' $G$ has. Many community detection algorithms are based on a modularity maximisation principle~\cite{louvain, cohen2020power, traag2019louvain}, with a wide range of applications, from protein detection to connections between web sites. See \cite{For10,POM09} for surveys on modularity.\\ 

On the other hand, expander graphs are important objects in graph theory and theoretical computer science. An expander graph has a wide variety of properties: for instance, the random walk mixes very fast~\cite{AKS87,Gil98}, and the eigenvalues of its Laplacian are well-separated~\cite{Al86, Che70}. In computer science, they are used for clustering with the expander decomposition technique, see for instance~\cite{KVV04}. We refer to~\cite{HLW06} for a survey of expander graphs and their applications.\\ 

The modularity value $q^*(G)$ of a graph $G$ is robust to small perturbations in the edge-set: changing an $\eps$ proportion of the edges in the graph changes the modularity value by at most~$2\eps$~\cite{ERmod}. In contrast, the property of being an expander, and the expansion constant, is influenced by even very small regions in the graph: for example adding a disjoint edge causes the expansion constant to drop to zero. %Thus by adding a disjoint edge for any error~$\eps$ if a $\delta$-expander graph $G$ has a large enough number of edges we can construct $G'$ with expansion constant $0$ such that its modularity value is within $\eps$ of $G$. \\
However, the property of having large expander subgraphs is robust to changes in the edge-set, see also the discussion in~\cite{chakraborti2022well, krivelevich2019expanders}.  Here, by `large' we mean containing at least a constant fraction of the edges of the graph. It turns out this notion of containing large induced expander subgraphs will be the right expander property of graphs to consider when relating modularity and expansion. The question of the presence of large expanders in graphs has been considered in previous works~\cite{bottcher2010bandwidth,  Kri18, krivelevich2019expanders, LS20}.\\

Here is a quick guide to the results, which will be presented after the definitions.  Theorem~\ref{thm} shows that modularity bounded below 1 implies having a large expander subgraph and vice versa; with detailed upper and lower bounds given in Propositions~\ref{prop.exptoqub} and~\ref{prop.noexptoqlb}. % - which together imply Theorem~\ref{thm}.\\
Our other main result, Theorem~\ref{thm:reslimit} shows that whether a connected component~$H$ in host graph~$G$ is split (or kept as one part) in an optimal partition of $G$ is characterised by the ratio $e(H)/e(G)$ and an edge expansion constant $\hh_H$ of $H$ - see Section~\ref{subsec.defns}. This %has algorithmic implications and 
extends the `resolution limit' known to exist for modularity.\\

 The results are stated in Section~\ref{subsec.results} and proven in Section~\ref{sec.proofs}. In Section~\ref{sec.constructions} we give families of examples showing the tightness of Theorem~\ref{thm} and part of Proposition~\ref{prop.exptoqub}, the other part remains open, see Section~\ref{sec.concl}. %the concluding remarks.  %Section~\ref{sec.constructions} gives details of the constructions and proves they have the properties claimed.

\needspace{5\baselineskip}
\subsection{Definitions}\label{subsec.defns}

\subsubsection*{Graph definitions}

Given a graph $G$, for disjoint sets $A$ and $B$ of vertices, let $e_G(A)$ be the number of edges within~$A$, let $e_G(A,B)$ be the number of edges between $A$ and $B$, and let the \emph{volume} $\vol_G(A)$ be the sum over the vertices $v$ in $A$ of the degree $d_v$. We will sometimes drop the subscript if it is clear from context. We restrict our attention to graphs $G$ with at least one edge (that is, \emph{non-empty} graphs), usually without comment.\\

Our graphs may have multiple or weighted edges and loops. For such a graph $G$, $e_G(A)$ and $e_G(A,B)$ are the sums of the weights of the corresponding edges. Similarly the weighted degree $d_v$ of a vertex $v$ is the sum of the weights of the incident edges, with loops counting twice to the sum. The volume $\vol_G(S)$ of a set $S$ of vertices is the sum of the weighted degrees of the vertices in $S$.

\subsubsection*{Expansion definitions}

Let us now introduce the property of graph expansion and two ways of measuring it: relative to the volume %$\vol$, 
of the smaller set or to a product of volumes. Given a graph~$G$ the \emph{conductance} or \emph{Cheeger constant} $h_G$ is defined as follows. Write $\bar{A}$ to denote $V(G)\backslash A$ and let
\begin{equation}\label{def.h}
h_G(A)=\frac{e_G(A, \bar{A})}{\min\{\vol_G(A), \vol_G(\bar{A})\}} \;\;  \mbox{for } \emptyset \neq A \; \subsetneq \;  V(G), \;\;\; \mbox{and} \;\;\; h_G=\min_{A} h_G(A)\ .
\end{equation}
%
%\begin{equation}\label{def.h} h_G(A)=\frac{e_G(A, \bar{A})}{\min\{\vol_G(A), \vol_G(\bar{A})\}} \;\;\;\;\;\;\; \mbox{and} \;\;\;\;\; h_G=\min_{A \subset V(G)} h_G(A).  \end{equation}
We say that $G$ is a \emph{$\delta$-expander} for any $0<\delta \leq h_G$. Observe that $0 \leq h_G \leq 1$, $h_G=0$ iff $G$ is disconnected, and $h_G=1$ iff $G$ is $K_3$ or $K_{1,t}$ for some $t \geq 1$, see Remark~\ref{rem.hbounds}.\\

Now define the following variant $\hh_G$ of graph expansion, by replacing the minimum with a product and normalising. Let
\begin{equation}\label{def.hh}
\hh_G(A)=\frac{e_G(A, \bar{A})}{\vol_G(A) \vol_G(\bar{A})}\vol(G) \;\;\; \mbox{for } \emptyset \neq A \; \subsetneq \;  V, \;\;\;\; \mbox{and} \;\;\;\; \hh_G=\min_{ A } \hh_G(A). 
\end{equation}
We say $G$ is a \emph{$\hdelta$-expander-by-products} for any $0<\hdelta \leq \hh_G$. Observe that $0 \leq \hh_G \leq 2$, $\hh_G=0$ iff $G$ is disconnected, and $\hh_G=2$ iff $G$ is $K_2$, see Remark~\ref{rem.hbounds}. Further we always have %$\frac12 \hh_G \leq h_G \leq \hh_G$. % for every non-empty  graph $G$.
 \begin{equation} \label{eqn.hGhhG}
\tfrac12 \, \hh_G \leq h_G \leq \hh_G  
 \end{equation}
(and $h_G < \hh_G$ if $G$ is connected).\\ 

This notion $\hh_G$ of expansion-by-products was used for example in the paper by Kannan, Lov{\'a}sz and Simonovits~\cite{kannan1995isoperimetric} as they found it a more natural definition to relate expansion and Markov chains. We will find that under the notion of expansion-by-products we have a tight upper bound for modularity in Proposition~\ref{prop.exptoqub}. Also the value $\hh_H$ of a connected component~$H$ in host graph~$G$, together with the relative size $e(H)/e(G)$, characterise whether or not the vertex set of the component is split in a modularity optimal partition of~$G$, see~Theorem~\ref{thm:reslimit}.

\subsubsection*{Modularity definitions}

For a graph $G$, we assign a modularity score $q_\cA(G)$ to each vertex partition (or `clustering')~$\cA$. The modularity $\q(G)$ of $G$, sometimes called the `maximum modularity' of $G$, is defined to be the maximum of these scores over all vertex partitions. \\

Modularity was introduced in Newman \& Girvan~\cite{NewmanGirvan}. Let $G$ be a graph with vertex set $V$ and with $m\geq 1$ edges. For a partition $\cA$ of $V$, % the vertices of~$G$, 
the modularity score of $\cA$ on $G$ is 
\begin{equation*} q_\cA(G) = \frac{1}{m}\sum_{A \in \cA} e(A) - \frac{1}{4m^2}\sum_{A\in \cA} \vol (A)^2; \end{equation*}
and the (maximum) modularity of $G$ is $\q(G)=\max_\cA q_{\cA}(G)$, where the maximum is over all partitions~$\cA$ of~$V$. We will write modularity as the difference of two terms, the \emph{edge contribution} or \emph{coverage} $q^E_\cA(G)=\tfrac{1}{m}\sum_{A \in \cA} e(A)$, and the \emph{degree tax} $q^D_\cA(G)=\tfrac{1}{4m^2}\sum_{A \in \cA} \vol(A)^2$.  For a graph $G$ with no edges, by convention we set $q_\cA(G)=0$ for every vertex partition $\cA$ and thus $\q(G)=0$.

%%%%%%%%%%%%%%%%%%%%%%%%%%%%%%%%%%%%%%

\needspace{5\baselineskip}
\subsection{Statement of Results}\label{subsec.results}

\subsubsection*{Modularity of $G$ and expansion of subgraphs $H$}

Recall that $0 \leq q^*(G) <1$, and the closer $q^*(G)$ is to $1$ the more community structure $G$ is considered to have.
Our main result, Theorem~\ref{thm}, is that having low modularity implies having a large expander induced subgraph, and vice versa.\\

We shall use the following function $f$, see Figure~\ref{fig.f}.  For $0 < \alpha \leq 1$ let $f(\alpha)$ be the maximum value of $\sum_i x_i^2$ where $0 \leq x_i \leq \alpha$ for each $i$ and $\sum_i x_i=1$; so 
\begin{equation} \label{eqn.f}
f(\alpha)=\alpha^2\lfloor1/\alpha\rfloor + (1-\alpha \lfloor1/\alpha \rfloor)^2 \;\; \mbox{ for }  0< \alpha \leq 1\,. 
\end{equation}

Notice that $f$ is continuous and increasing, 
$f(\alpha) = \alpha$ when $\alpha=1/k$ for some integer $k$, and always $\alpha - \alpha^2/4 \leq f(\alpha) \leq \alpha$, so for small $\alpha$ we have $f(\alpha)\sim \alpha$. 
\begin{figure*}
    \centering
    \hspace{-3mm}\begin{subfigure}[t]{0.5\textwidth}
        \centering
        \includegraphics[scale=0.7]{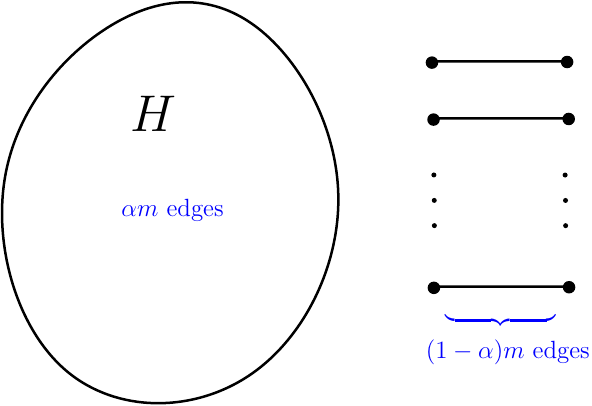}
        \caption{$G_H$ with $\q(G_H)\!\geq\! 1\!-\!\alpha^2\!-\!O(m^{-1}).$
        }
    \end{subfigure}%
    ~\hspace{3mm}~
    \begin{subfigure}[t]{0.5\textwidth}
        \centering
        \includegraphics[scale=0.7]{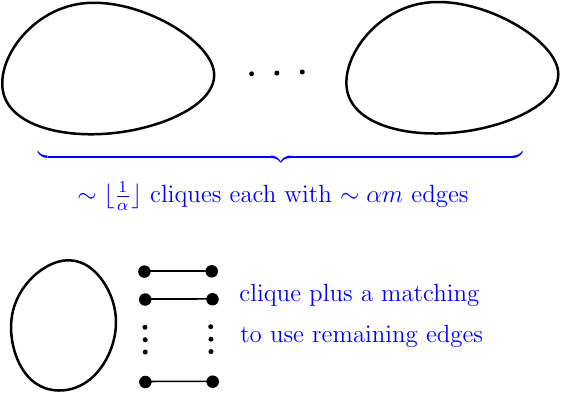}
        \caption{$G_\alpha$ with $\q(G_\alpha)\!=\!1-\!f(\alpha)\!+\!O(m^{-\frac{1}{2}}).$
        }
    \end{subfigure}
\caption{Example constructions for Theorem~\ref{thm}, see Section~\ref{sec.example_thm1_1} for details.}\label{fig.constructions}
\end{figure*}
\needspace{5\baselineskip}
\needspace{5\baselineskip}
\begin{theorem}\label{thm}
Let $0< \alpha <1$ and let $\eps>0$.
\begin{enumerate}[label=(\alph*)]
    \item\begin{enumerate}[label=(\roman*)]
        \item If the graph $G$ contains a subgraph $H$ with $e(H) \geq \alpha e(G)$ which is an $\alpha$-expander (that is $h_H \geq \alpha$) then $\q(G)\leq 1-\alpha^2$. 
        \item Conversely, for all sufficiently large $m \in \NN$, for \emph{each} %\footnote{and in particular, we can take $H$ to be an expander.} 
        graph $H$ with $e(H) \leq \alpha m$ there exists an $m$-edge graph $G_H$ containing $H$ as an induced subgraph such that $\q(G_H) \geq 1-\alpha^2-\eps$. 
%(In particular we can choose $H$ an expander.)
        \end{enumerate}
    \item\begin{enumerate}[label=(\roman*)]
        \item There exists $\delta>0$ such that every graph $G$ with $\q(G) < 1-f(\alpha)-\eps$ contains an induced subgraph $H$ with $e(H) \geq \alpha e(G)$ which is a $\delta$-expander (that is $h_H\geq \delta)$. 
        \item Conversely, for all sufficiently large $m \in \NN$ there is an $m$-edge graph $G_\alpha$ with $\q(G_\alpha)< 1-f(\alpha)+ \eps$ such that each subgraph of $G_\alpha$ with at least $\alpha\, m$ edges is disconnected.
\end{enumerate} 
\end{enumerate}
\end{theorem}

Constructions for the two `converse' statements (a)(ii) and (b)(ii) in the theorem are illustrated in Figure~\ref{fig.constructions}, with full details and proofs of the bounds deferred to Section~\ref{sec.example_thm1_1}. 
Notice that the construction for (a)(ii) is particularly simple - given $H$ we may form $G_H$ by adding $e(G)-e(H)$ disjoint edges to $H$. Notice also that in (a)(ii) we may for example take $H$ to be an $\alpha$-expander. (by products). \\

The first statements (a)(i) and (b)(i) in Theorem~\ref{thm} follow immediately from the next two propositions, which give bounds on the dependence between the modularity value of $G$ and the relative size and expansion constant of subgraphs of $G$. 
Recall from Section~\ref{subsec.defns} that $h_{H}$ denotes the usual notion of conductance, while $\hh_H$ is a measure of edge expansion normalised by the product of the volumes of the two parts.

\needspace{5\baselineskip}
\begin{proposition} \label{prop.exptoqub}
    \begin{enumerate}[label=(\roman*)]
        \item 
        Let $0 < \alpha \leq 1$, let $G$ be a graph and let $H$ be a subgraph with relative size $e(H)/e(G) \geq \alpha$. Then
        \[ 
        q^*(G) \leq 1- \alpha \min \{\hat{h}_H,\alpha\}. 
        \]
        \item{Conversely, for all $\eps>0$, $0\leq a \leq 1$, $0 \leq \hdelta \leq 1$, there exists a graph $G$ with induced subgraph $H$ of relative size $\alpha=e(H)/e(G)$ such that $\,|\alpha-a|, |\hh_H-\hdelta| <\eps$ and 
        \[\q(G) \geq 1-\alpha \min\{ \hh_H , \alpha\}-\eps.\] }
    \end{enumerate}
\end{proposition}
In Proposition~\ref{prop.exptoqub}, the converse statement (ii) shows that the inequality in (i) is tight. Since $h_H \leq \hh_H$ by (\ref{eqn.hGhhG}), from (i) we have
\[ \q(G) \leq 1 - \alpha \min\{\hh_H, \alpha\} \leq 1- \alpha \min\{h_H, \alpha\};\]
and in the special case $\alpha=1$ we obtain
%\begin{equation*}% \label{cor.prop1.2}
$\q(G) \leq 1 - \min\{\hh_G, 1\} \leq 1-h_G$ % \,,
%\end{equation*
see Section~\ref{sec.concl} for a discussion. 
%
%Recall that the function $f$ is defined in~(\ref{eqn.f}). %todo reinclude?
\begin{figure}
    \centering
    \includegraphics[scale=0.3]{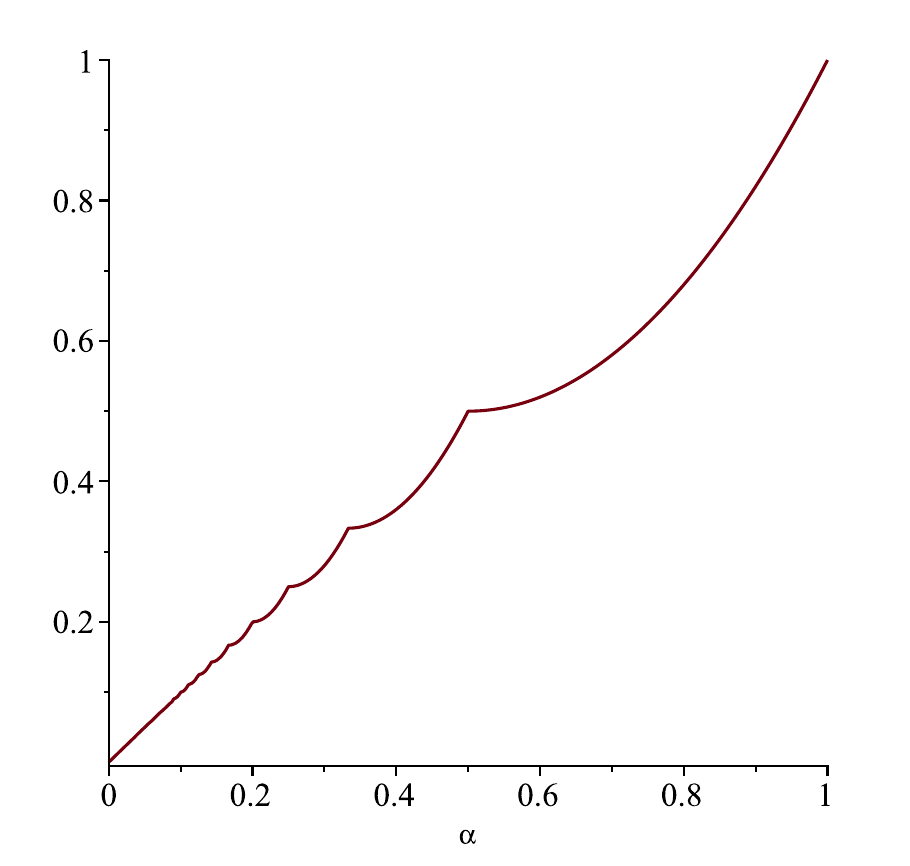}
    \caption{The function $f$, see \eqref{eqn.f}.}\label{fig.f}
\end{figure}
\needspace{5\baselineskip}
\begin{proposition} \label{prop.noexptoqlb}
Let $0 < \alpha, \delta \leq 1$, let $G$ be a graph, and suppose that each induced subgraph~$H$ %of $G$ 
with relative size $e(H) / e(G) \geq \alpha$ has $h_H < \delta$. Then 
\[ \q(G) > 1-f\left(\min\{1,\alpha+\tfrac 3 2 \delta\}\right) -  \tfrac{3}{2}\, \delta \lceil \log_2(1/\alpha)\rceil.\]
\end{proposition}
To prove the first statements (a)(i) and (b)(i) in Theorem~\ref{thm} from these two propositions, one just needs to assume $h_H \geq \alpha$ (since then also $\hh_H \geq \alpha$) in Proposition~\ref{prop.exptoqub}(i); and $\delta$ small enough in Proposition~\ref{prop.noexptoqlb}, and use the fact that $f$ is continuous.

%%%%%%%%%%%%%%%%%%%%%%%%%%%%%%%%%%%%%%%%%%%%%%%%%%%%%%

\subsubsection*{When should we split a component?}
In an optimal partition of any graph $G$ without isolated vertices, for each part $A$ the induced subgraph $G[A]$ is connected~\cite{nphard}. (If not, say $e(A', A\backslash A')=0$ for some $\emptyset \neq A' \subset A$, then replacing part $A$ by parts $A'$ and $A \backslash A'$ strictly increases the modularity score: the edge contribution stays the same while the degree tax is strictly decreased.) This implies that each connected component $H$ is partitioned separately from the rest of the host graph.\\ 

We find that whether an optimal partition of $G$ splits $H$ into multiple parts or keeps the vertices of $H$ together in one part is characterised by $\alpha$ and $\hh_H$. Let $\OPT$ be the set of partitions $\cA$ for $G$ such that $q_\cA(G)=\q(G)$. 
\needspace{5\baselineskip}
\begin{theorem}\label{thm:reslimit}
Let $H$ be a connected component in a graph $G$ and let $\alpha$ be the relative size~$e(H)/e(G)$.
\begin{itemize}
    \item If $\alpha < \hh_H$ then $\forall \cA \in \OPT$, $H$ is not split. 
    \item If $\alpha > \hh_H$ then $\forall \cA \in \OPT$, $H$ is split. 
    \item If $\alpha = \hh_H$ then $\exists \,\cA, \cA' \in \OPT$ such that $H$ is not split in $\cA$ and $H$ is split in $\cA'$.
\end{itemize}

\end{theorem}

Note that $\hh_H$ is a minimum of a function over cuts or bipartitions of $V(H)$, rather than over all partitions of $V(H)$ including those containing more than two parts. Hence it is interesting that $\alpha$ and $\hh_H$ are enough to determine when a connected component is split in an optimal partition of the host graph.\\  

Recall the original resolution limit bound, in~\cite{FortBart2008}, that any modularity optimal partition will not split a connected component $H$ in graph $G$ if $e(H) < \sqrt{2 e(G)}$. This is implied by Theorem~\ref{thm:reslimit}. To see this, note first that $e(U, V(H) \backslash U) \geq 1$ for all $\emptyset \neq U  \subsetneq  V(H)$ (since~$H$ is connected). Also $\max_U \vol(U)(\vol(H)-\vol(U)) \leq \vol(H)^2/4 = \vol(H) e(H)/2$. Hence 
$\hh_H \geq 2/e(H)$. By the theorem, $H$ will not be split if $e(H)/e(G) < 2/e(H)$ (since then $e(H)/e(G) < \hh_H$). Rearranging, we see that $H$ will not be split if $e(H) < \sqrt{2 e(G)}$, which recovers the resolution limit bound.\\

Theorem~\ref{thm:reslimit} also shows that connected components with modularity zero will not be split in any optimal partition of~$G$.
\begin{corollary}\label{cor.0modpart}
Let the graph $G$ have no isolated vertices, and let $H$ be a connected component of $G$ with $e(H) < e(G)$. Let $\cA$ be an optimal partition for $G$, that is, $q_{\cA}(G)=\q(G)$. If $\q(H) =0$ then $V(H)$ is a part in $\cA$.
\end{corollary}
To see why Corollary~\ref{cor.0modpart} follows from Theorem~\ref{thm:reslimit} (and Proposition~\ref{prop.q0}), note by Proposition~\ref{prop.q0} that $\q(H)=0$ implies $\hh_H \geq 1$. Now since $e(H)<e(G)$ we have that $\alpha=e(H)/e(G)<\hh_H$ and thus by Theorem~\ref{thm:reslimit} the component $H$ is not split in any optimal partition of $G$.\\

This corollary tells us, for example, that each clique in Figure~\ref{fig.constructions} forms its own part in any optimal partition.

%%%%%%%%%%%%%%%%%%%%%%%%%%%%%%%%%%%%%%%%%%%%%%%%%%%%%

\subsubsection*{Modularity near one}
The next result says roughly that $\q(G) \approx 1$ if and only if each `large' subgraph $H$ of $G$ has $h_H \approx 0$.
\needspace{5\baselineskip}
\begin{corollary}\label{cor.q0}$\mbox{ }$
    \begin{enumerate}[label=(\alph*)] %itemize with was [(a)] etc.
        \item	For all $\alpha, \delta>0$ there exists $\eps>0$ such that if the %non-empty
            graph $G$ has a subgraph $H$ with at least $\alpha\, e(G)$ edges which is an $\delta$-expander then $q^*(G) < 1-\eps$ (indeed we may take $\eps = \alpha \min\{\alpha, \delta\}$). %, assuming $\alpha<1$)
        \item	For all $\eps>0$ there exists $\alpha>0$ such that if the graph $G$ has no subgraph $H$ with at least $\alpha\, e(G)$ edges which is an $\alpha$-expander then $q^*(G) > 1-\eps$.
    \end{enumerate}
\end{corollary}

This last result shows that, given a sequence $G_m$ of $m$-edge graphs, as $m \to \infty$ we have $q^*(G_m) \to 1$ if and only if
$\max_{H \subseteq G_m} \min \{ e(H)/m, h_H\} \to 0$. (This statement remains true if instead of maximising over all subgraphs $H$ of $G$ we maximise over all topological minors~$H$ of $G$, see \cite{LS20}.) Thus we have a characterisation of `maximally modular' graphs $G_m$ as defined in~\cite{modgraphclasses}. %{\Cc Loosely speaking, $\q(G_m)=1-o(1)$ if and only if each subgraph $H$ satisfies $e(H)=o(m)$ or $h_H=o(1)$.} 

%%%%%%%%%%%%%%%%%%%%%%%%%%%%%%%%%%%%%

\subsubsection*{Modularity zero}
Now consider the other extreme, when $\q(G)=0$.  It is known that this holds for complete graphs~\cite{nphard} and complete multipartite graphs \cite{bolla2015spectral, majstorovic2014note} and for all graphs constructed by taking a complete graph on $n$ vertices and deleting at most $n/2$ edges~\cite{vdense}. 
We can characterise when $\q(G)=0$ in terms of a form of graph expansion.  This can be expressed in a way now involving not the minimum of $e(A)$ and $e(\bar{A})$ but their geometric mean, or involving a product of $\vol(A)$ and $\vol(\bar{A})$. Here $\bar{A}$ denotes~$V(G) \setminus A$.
\needspace{5\baselineskip}
\begin{proposition} \label{prop.q0}
    For a %non-empty
    graph $G$, the following three conditions are equivalent:
    \begin{enumerate}[label=(\alph*)]
        \item $\; \q(G)=0$
        \item $\; e(A,\bar{A}) \geq 2 \sqrt{e(A) e(\bar{A})}\;$ for all $A \subseteq V(G)$
        \item $\; e(A,\bar{A}) \geq \vol(A) \vol(\bar{A})/ \vol(G)\;$ for all $A \subseteq V(G)$\,, i.e.\ $\hh_G \geq 1$.
    \end{enumerate}
\end{proposition}
By Proposition~\ref{prop.q0}, if the 
graph $G$ satisfies $\q(G)=0$, then 
\[ e(A,\bar{A}) \geq \vol(A) \vol(\bar{A})/ \vol(G) \geq \min\{ \vol(A), \vol(\bar{A})\} \cdot \tfrac12  \;\; \mbox{ for all } A \subseteq V(G)\,, \]
and so we have the following corollary concerning the conductance $h_G$.
\needspace{1\baselineskip}
\begin{corollary}\label{cor.zeromodimplieshalf}
If the graph $G$ satisfies $\q(G)=0$, then $h_G \geq 1/2$.
\end{corollary}

The reverse statement to the corollary does not hold. See Lemma~\ref{lem.windmill} and Figure~\ref{fig.windmill} for an example `windmill' graph with conductance $1/2$ and positive modularity. Proposition~\ref{prop.q0} also yields a short proof that complete multipartite graphs have modularity zero, see~\cite{vdense}.

\needspace{5\baselineskip}
\subsection{Relation to existing results and our contribution}\label{sec.relation}
\paragraph*{Natural conditions to guarantee a large expander subgraph} 
What properties of the graph imply that it will have an expander subgraph on a linear proportion of the edges? This question is asked in~\cite{chakraborti2022well} which shows that having a positive proportion of vertices such that a random walk starting from them is well-mixing is sufficient, indeed, it implies the graph contains an almost spanning subgraph which is an expander. We show that modularity bounded away from 1 is another such condition which guarantess a linear sized expander subgraph.

%%%%%

\paragraph*{Relation to spectral properties of subgraphs.}
Our results relate the expansion properties of subgraphs $H$ of $G$ to the modularity of $G$, but how about spectral properties of $H$, what do these imply about the modularity of~$G$?\\

Define the normalised Laplacian of graph $H$ to be $\mathcal{L}_H= I - D^{-1/2}A D^{-1/2}$ where $A$ is the adjacency matrix of $H$, $d_i$ is the degree of vertex $i$, and $D^{-1/2}$ is the diagonal 
matrix with $i$-th diagonal entry~$d_i^{-1/2}$.  The spectral gap of $H$ is $\bar{\lambda}_{H} = \max_{i \neq 0} |1-\lambda_i|$ where $0=\lambda_0 \leq \lambda_1 \leq \cdots$ are the eigenvalues of~$\mathcal{L}_H$. Always $0 \leq \bar{\lambda}_H \leq 1$. 
An expander-mixing lemma, see Corollary~5.5 of \cite{chung1997spectral}, says that for any graph~$H$ we have $\hh_H \geq 1-\bar{\lambda}_H$, and thus the following result is an immediate corollary of Proposition~\ref{prop.exptoqub}.
\needspace{2\baselineskip}
\begin{corollary}\label{cor.spectralub}
Let the graph $G$ have  a subgraph~$H$ with spectral gap $\bar{\lambda}_H$, and with $e(H) \geq \alpha\, e(G)$ where $\alpha>0$. Then 
\[ q^*(G) \leq 1- \alpha \min \{\alpha, 1-\bar{\lambda}_H \} .\]
\end{corollary}

Note that in the case $\alpha=1$, that is when the graph $G$ itself has spectral gap $\bar{\lambda}$ this recovers the results $\q(G)\leq \bar{\lambda}$ in \cite{fasino2014algebraic, van2010spectral} (regular graphs) and \cite{ERmod} (general graphs). Spectral results were used to prove upper bounds on the modularity in~\cite{treelike, ERmod, pralat} and lower bounds in~\cite{lason2021modularity} - see Proposition~\ref{prop.noexptoqlb2}. 

%%%%%

\paragraph*{Graphs with asymptotically maximal modularity}
De Montgolfier, Soto and Viennot~\cite{modgraphclasses} defined the notion of maximally modular classes of graphs as those for which $\q(G)\rightarrow 1$ as \mbox{$e(G)\rightarrow \infty$}, and showed hypercubes and tori are maximally modular, as well as $n$-vertex trees with maximum degree $\Delta=o(n^{1/5})$. This was extended to the class of trees with $\Delta=o(n)$~\cite{treelike}, to the class of graphs with treewidth such that the product of treewidth and max degree is $o(e(G))$~\cite{treelike}, and to the class of minor-free graphs with maximum degree $\Delta=o(e(G))$~\cite{lason2021modularity}. Our paper extends this to graphs $G$ where any expander subgraph $H$ satisfies $e(H)=o(e(G))$, see Corollary~\ref{cor.q0}.

\paragraph*{Graphs with bounds on modularity values}
Given that the most popular clustering algorithms for large networks are modularity-based~\cite{popular} it is important to build up our theoretical knowledge on modularity. Finding the modularity for graph classes helps us understand the behaviour of the modularity function. For a list of results see the table in~\cite{ERmod}, noting that since then it has been established in \cite{lichev_mitsche_2022} that whp random cubic graphs have modularity value in the interval $[0.667, 0.79998]$, and in~\cite{lason2021modularity} that graphs which have bounded genus and maximum degree $o(n)$ have modularity asymptotically~1. Also, the modularity of the stochastic block model is considered in~\cite{bickel2009nonparametric,cohen2020power,koshelev2023modularity} and of random geometric graphs in~\cite{davis2018consistency}. Our paper contributes to this line of work by proving upper and lower bounds on modularity in terms of the relative size and expansion of subgraphs.

\paragraph*{Key contributions of this paper}
We relate modularity and graph expansion by formulating and proving upper and lower bounds for the modularity of graphs given the sizes and expansion coefficients of their subgraphs. In the constructions, Section~\ref{sec.constructions}, we establish their modularity values and the expansion-by-products of their subgraphs. This yields the tightness of the upper bounds on~$q^*$ in Proposition~\ref{prop.exptoqub} for all values of $\alpha$ and~$\hh_H$. Previously the special cases $\q(G)\leq 1-h_G$ and $\q(G)\leq \bar{\lambda}_G$, as mentioned earlier, were known.  Recall that robustness results of~\cite{ERmod} show that, if~$G$ has a subgraph $H$ then
$\q(G)\leq \q(H) + 2e(G\backslash H)/e(G)$ which yields the
upper bound 
$\q(G)\leq \bar{\lambda}_H+2(1-\alpha)$ where $\alpha=e(H)/e(G)$ (this was used to prove that whp $\q(G_{n, d/n})=\Theta(d^{-1/2})$). 
Lower bounds on $\q$ in terms of expansion of subgraphs were used in~\cite{lason2021modularity} and we improve these - see Proposition~\ref{prop.noexptoqlb2}.\\

The other key contribution is to deepen the understanding of the well-known resolution limit for modularity established by Fortunato and Barth\'elemy~\cite{FortBart2008}. Loosely this says that modularity cannot pick up the community structure of any connected component of an $m$-edge graph if the component has fewer than $\sqrt{2m}$ edges. Theorem~\ref{thm:reslimit} determines the scale at which the community structure of a connected component becomes detectable by modularity, as a function of an expansion coefficient of the component.

%%%%%%%%%%%%%%%%%%%%%%%%%%%%%%%%%%%%%%%%%%%%%%%%%%%
%%%%%%%%%%%%%%%%%%%%%%%%%%%%%%%%%%%%%%%%%%%%%%%%%%%

\section{Proofs}\label{sec.proofs}

In this section we prove Propositions~\ref{prop.exptoqub},~\ref{prop.noexptoqlb} and~\ref{prop.q0}.  Note that this yields all the results presented in Section~\ref{sec.intro}, except for the `conversely' statements (a)(ii) and (b)(ii) in Theorem~\ref{thm} and Proposition~\ref{prop.exptoqub}(ii) which are all based on proving properties of constructions, see Section~\ref{sec.constructions}.\\

For completeness, we also include a remark on the upper bounds on $h_G$ and $\hh_G$ claimed in the introduction, as we could not find a reference.
\begin{remark}\label{rem.hbounds} 
It is easy to see that $h_G \leq 1$. Let us show that $h_G = 1$ iff $G$ is $K_3$ or $K_{1,t}$ for $t\geq 1$.
First note that $h_G=h_G(A)$ for some $A$ with $G[A]$ connected and $\vol_G(A) \leq \vol_G(\bar{A})$.
We may now see that $h_G = 1$ iff $G$ does not contain two disjoint edges. If $G$ contains two disjoint edges $uv$ and $xy$ say, and $\vol(\{u,v\}) \leq \vol(\{x,y\})$ then 
$h_G(\{u,v\}) <1$.
If $G$ does not contain two disjoint edges then for any $|A|\geq 2$ with $G[A]$ connected we have $\vol_G(A)>\vol_G(\bar{A})$, so $h_G=h_G(A)$ for some $|A|=1$ and thus $h_G=1$ and we are done.\\

To show $\hat{h}_G \leq 2$ with equality iff $G$ is $K_2$, one can argue as follows. 
For nonempty $A \subsetneq V$ with say $\vol_G(A) \leq \vol_G(\bar{A})$, since  $\vol(G) / \vol_G(\bar{A}) \leq 2$  we have
\[ \hat{h}_G(A)  \leq \frac{ e_G(A, \bar{A}) }{ \vol_G(A) } \cdot 2  \leq 2 \]
with strict inequality unless $\vol_G(A)=\tfrac12  \vol(G)$. But the only graph (without isolated vertices) such that $\vol_G(A)= \tfrac12 \vol(G)$  for each nonempty $A \subsetneq V$ is $K_2$ and we are done.
\end{remark}

%%%%%%%%%%%%%%%%%%%%%%%%%%%%%%%%%%%%%%%%%%%%%%%%%%%%%%

\needspace{5\baselineskip}
\subsection{Proof of Proposition~\ref{prop.exptoqub}}
\begin{proof}[Proof of Proposition~\ref{prop.exptoqub}]
Let $\cA$ be a partition of $V(G)$, let $\cB=\cB(\cA, H)$ be the induced partition of $V(H)$ and define $x_B=\vol_H(B)/\vol(H)$ for parts $B\in \cB$. Let $\hdelta := \hh_H$. We will prove the statement 
\begin{equation}
    \label{eq.detailed_upperbound}
    q_\cA(G) \leq 1 - \alpha\hdelta + \alpha(\hdelta -\alpha)\sum_{B \in \cB} x_B^2.
\end{equation}

To see that \eqref{eq.detailed_upperbound} implies the proposition note $0 < \sum_B x_B^2 \leq 1$ and hence $q_\cA(G)\leq 1 - \alpha \min\{\hh_H, \alpha\}$. Since $h_H \leq \hh_H$ for any graph $H$ this implies $q_\cA(G)\leq 1 - \alpha \min\{h_H, \alpha\}$.\\

As $H$ is a $\hdelta$-expander (as defined above), for each $B\in \cB$ 
\[ e_H(B,V(H) \backslash B)/\vol(H) \geq \hdelta x_B(1-x_B).\]
Thus, 
\[ q_\cA^E(G) 
\leq 1 - \frac{\alpha}{\vol(H)} \sum_{B \in \cB} e_H(B,V(H) \backslash B)
\leq  1- \alpha\hdelta \sum_{B \in \cB} x_B(1-x_B) 
\]
and hence $q_\cA^E(G)\leq 1 - \alpha\hdelta (1- \sum_{B \in \cB} x_B^2)$. For the degree tax
\[ q_\cA^D(G) \geq \frac{\alpha^2}{\vol(H)^2} \sum_{B \in \cB} \vol_H(B)^2 = \alpha^2\sum_{B\in \cB} x_B^2\]
which establishes the statement \eqref{eq.detailed_upperbound}, and thus completes the proof of Proposition~\ref{prop.exptoqub}.\end{proof}

\needspace{5\baselineskip}
\subsection{Proof of Proposition~\ref{prop.noexptoqlb}}
We shall use two preliminary lemmas in the proof of Proposition~\ref{prop.noexptoqlb}.
For a vertex partition~$\AA$ we let $\partial_G(\AA)$ denote the number of edges between the parts of $\AA$.\\ % todo So we may write the edge contribution $q^E_\AA(G) = 1 - \partial_G(\AA)/e(G)$. \\

It will be convenient to use a different notion of expansion for the proof of Lemma~\ref{lem.new1}. For a graph~$H$, define  its \emph{expansion-by-edges}, $h'_H$, by taking the edge boundary of sets relative to the number of edges inside the set rather than the volume of the set. Let 
\begin{figure}
\centering
\includegraphics[scale=0.5]{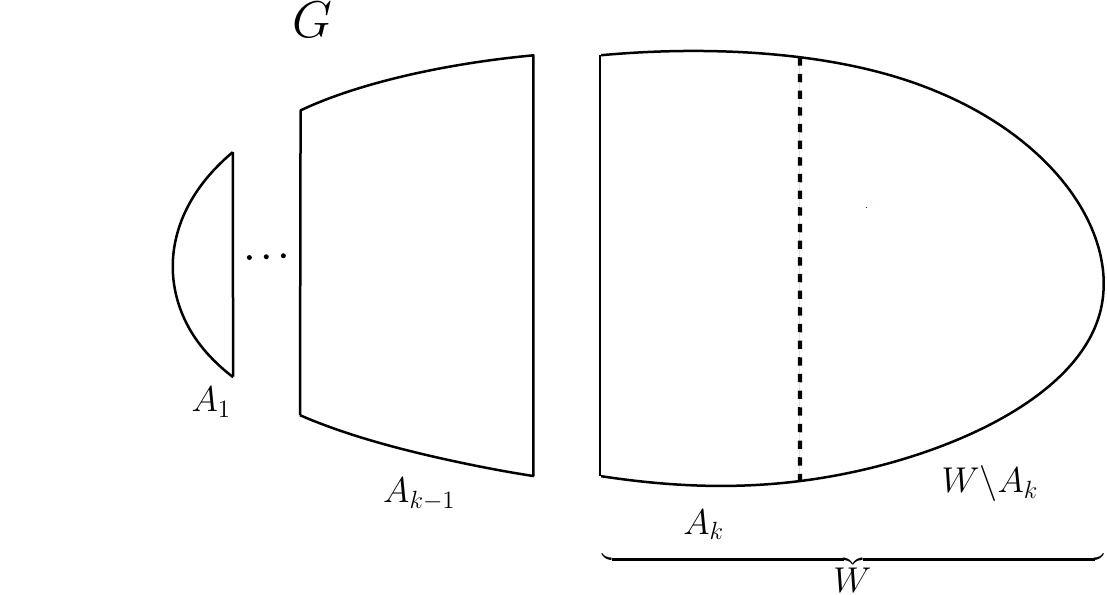}
\caption{Step $k$ for the algorithm in Lemma~\ref{lem.new1}. Since in $G$ any induced subgraph $H$ with size~$e(H) > e_0$ is not a $\delta'$-expander-by-edges (see definition on line~\eqref{eq.def_expbyedges}): if $e(W) > \max\{ \tfrac{1}{2}e(G) , e_0 \}$ there is a vertex subset $A_k$ with 
$e(A_k)  \leq e(W\backslash A_k)$ and $e(A_k, W\backslash A_k)<\delta'e(A_k)$. }\label{fig.alg}
\end{figure}
\begin{equation}\label{eq.def_expbyedges}
h'_H(A)=\frac{e_H(A, \bar{A})}{\min\{ e_H(A), e_H(\bar{A})\}} \;\;\;\;\;\;\; \mbox{and} \;\;\;\;\; h'_H=\min_{A} h'_H(A)\,,
\end{equation}
where $\bar{A}$ denotes $V(H)\backslash A$. We say $H$ is a \emph{$\delta'$-expander-by-edges} for any $0<\delta' \leq h'_H$.\\ 

Since $\vol_H(A)=e_H(A,\bar{A}) +2e_H(A)$, for $0 \leq \delta < 1$ the inequality $e_H(A, \bar{A})\geq \delta\, \vol_H(A)$ is equivalent to $e_H(A,\bar{A})\geq \frac{2\delta}{1-\delta}\, e_H(A)$. %Also note that $\vol(A)-\vol(\bar{A}) = 2(e(A)-e(\bar{A}))$ so $\vol(A)<\vol(\bar{A})$ iff $e(A)<e(\bar{A})$.
 Hence, for $0 \leq \delta <1$
\[ e_H(A, \bar{A})\geq \delta\min\{ \vol_H(A), \vol_H(\bar{A})\}  \;\Leftrightarrow  \; e_H(A, \bar{A})\geq \tfrac{2\delta}{1-\delta} \min\{ e_H(A), e_H(\bar{A})\}.\] %todo - was , and so
Thus $H$ is a $\delta$-expander iff it is a $\delta'$-expander-by-edges where $\delta' = \tfrac{2\delta}{1-\delta}$, i.e.\ $h'_H=2h_H/(1-h_H)$. 
\needspace{5\baselineskip}
\begin{lemma}\label{lem.new1}
    Let $e_0 >0$ and $0<\delta< 1$; and let $G$ be a graph such that for all $U \subseteq V(G)$ with $e_G(U) > e_0$ the graph $G[U]$ is not a $\delta$-expander.  Then there is a partition $\AA$ of $V(G)$ such that 
    \begin{enumerate}[label=(\alph*)]
        \item each part $A$ of $\AA$  satisfies $e(A) \leq \max \{\frac12 e(G), e_0\}$, and 
        \item $\partial_G(\AA) \leq  \tfrac{\delta}{1+\delta} \tfrac{3+\delta}{2}\, e(G)$.
    \end{enumerate}
\end{lemma}
\begin{proof}
Let $\delta'= 2\delta/(1-\delta)$. 
Observe that for each $W \subseteq V(G)$ such that $\, e_G(W) > \max\{ \tfrac12 e(G), e_0 \}$, the graph $G[W]$ is not a $\delta'$-expander-by-edges, so there is a non-empty set $A \subset W$ such that $e_G(A) \leq e_G(W \setminus A)$ and $e_G(A, W \setminus A) < \delta'\, e_G(A)$. Thus in the following algorithm there will always be a set $A_k$ as required, and so the algorithm will succeed and will output a partition $\cA$ with $k+1 \geq 1$ parts $A_j$. See also Figure~\ref{fig.alg}.\vspace{0.1mm}

\begin{algorithm}
\SetAlgoNlRelativeSize{0}
% input/ouput names
\SetKwInOut{Input}{Input}
\SetKwInOut{Output}{Output}

    \Input{%
		\xvbox{4mm}{$\xvar{G}$} -- graph with no induced $\delta'$-expander-by-edges $H$ of size $e(H)> e_0$.\\
    }
    \Output{%
		\xvbox{4mm}{$\mathcal{A}$} -- vertex partition of $\xvar{G}$. 
    }

    \BlankLine % blank line for spacing
    % start of the pseudocode
    \xvbox{4mm}{$\xvar{W}$} $\leftarrow$ $V(\xvar{G})$, $\,$ 
    \xvbox{2.5mm}{$\xvar{k}$} $\leftarrow$ $0$ % \tcc*{}

    \While{$e_{\xvar{G}}(\xvar{W}\,) > \max\{\frac{1}{2}e(\xvar{G}),e_0\}$ }{
            \xvbox{3mm}{$\xvar{k}$} $\leftarrow$ $\xvar{k}+1$ % \tcc*{}

		pick $\varnothing \neq \xvar{A}_{\xvar{k}} \subset \xvar{W}$ with $e_{\xvar{G}}(\xvar{A}_{\xvar{k}} )\leq e_{\xvar{G}}(\xvar{W}\backslash \xvar{A}_{\xvar{k}})$ and $e_{\xvar{G}}(\xvar{A}_{\xvar{k}}, \xvar{W}\backslash \xvar{A}_{\xvar{k}} )<\delta' e_{\xvar{G}}(\xvar{A}_{\xvar{k}})$.
		
		\xvbox{3mm}{$\xvar{W}$} $\leftarrow$  $ \xvar{W} \backslash \xvar{A}_{\xvar{k}} $\;
  	} % end for k
	$\xvar{A}_{\xvar{k}+1}$ $\leftarrow$ $V(\xvar{G})\backslash(\xvar{A}_{1} \cup \ldots \cup \xvar{A}_{\xvar{k}})$  \;
\Return{ $\AA = \{\xvar{A}_1, \ldots \xvar{A}_{\xvar{k}+1} \} $}
\end{algorithm}

Clearly $e_G(A_j) \leq \frac12 e(G)$ for each $j=1.\ldots,k$ and $e_G(A_{k+1}) \leq \max\{\frac12 e(G), e_0\}$, so $\AA$ has property~(a) in the lemma.\\

If $k=0$ then $\cA$ is trivial and $\partial_G(\cA)=0$.
If $k \geq 1$ then
\begin{eqnarray*}
\partial_G(\AA) & = &
e_G(A_1,V(G) \setminus A_1) + e_G(A_2,V(G) \setminus (A_1 \cup A_2)) + \cdots + e_G(A_k,A_{k+1})\\
& \leq &
\delta' e_G(A_1) +  \delta' e_G(A_2) + \cdots + \delta' e_G(A_{k})\\
&=&
\delta' \, \sum_{j=1, \ldots, k}^{k} e_G(A_j) \;\; = \;\; \delta' (e(G) - e_G(A_{k+1}) - \partial_G(\AA)).
\end{eqnarray*}
The last two steps of the algorithm ensure that $ e_G(A_k)\leq e_G(A_{k+1})$ and $e_G(A_k \cup A_{k+1}) > \max\{ \tfrac12 e(G), e_0\}$. Also, we have 
\[e_G(A_k \cup A_{k+1})= e_G(A_k)+ e_G(A_{k+1})+e_G(A_k,A_{k+1})\leq e_G(A_k)+ e_G(A_{k+1})+\delta' e_G(A_k),\]
where the inequality comes from line $5$ of the algorithm.
Therefore, letting $x=e_G(A_k)$ and $y=e_G(A_{k+1})$, 
\[ x+ \delta' x + y \geq e_G(A_k \cup A_{k+1}) > \max\{ \tfrac12 e(G), e_0\} \geq \tfrac12 e(G).\]
Thus, since $x \leq y$, we have $(2 + \delta')y > \tfrac12 e(G)$, and so
$y > \tfrac1{4+2\delta'} e(G)$. Thus
\[ (1+\delta')\, \partial_G(\AA) \leq \delta' (e(G) -y) \leq \delta' \, \tfrac{3+2\delta'}{4+2 \delta'}\, e(G)\,.\]
Hence
\[ \partial_G(\AA) \leq \tfrac{\delta'}{1+\delta'}\, \tfrac{3+2\delta'}{4+2 \delta'}\, e(G)  = \tfrac{\delta}{1+\delta} \tfrac{3+\delta}{2} e(G)\,, \]
as required for property (b).\end{proof}

The following lemma records properties of the resulting partition after applying the algorithm in Lemma~\ref{lem.new1} to each part of a given partition.
Given a vertex partition $\AA$ of a graph~$G$, let $\maxin_G(\AA)$ denote $\max\{e_G(A) : A \in \AA \}$ and $\maxout_G(\AA)$ denote $\max\{e_G(A,\bar{A}) : A \in \AA \}$.  Thus property (a) in the last lemma says that $\maxin_G(\cA) \leq \max\{ \tfrac12 e(G), e_0 \}$.

\needspace{5\baselineskip}
\begin{lemma} \label{lem.new2}
Let $e_0 >0$ and $0<\delta< 1$; let $\rho =  \tfrac{\delta}{1+\delta} \tfrac{3+\delta}{2}$  (so $\delta < \rho < \tfrac32 \delta$); and let $G$ be a graph such that for all $U \subseteq V(G)$ with $e_G(U) \geq e_0$ the graph $G[U]$ is not a $\delta$-expander. 
Let~$\cA$ be a partition of $V(G)$. Then there is a partition $\cB$ refining $\cA$ such that 
    \begin{enumerate}[label=(\alph*)] %was itemize (a)
        \item $\partial_G(\cB) \leq (1-\rho)\,  \partial_G(\cA) + \rho\, e(G)\,$, and
        \item $\maxin_G(\cB) \leq \max \{\tfrac12 \maxin_G(\cA), e_0 \}$, and 
        \item $\maxout_G(\cB) \leq \maxout_G(\cA) + \rho \maxin_G(\cA)$.
    \end{enumerate}
\end{lemma}

\begin{proof}
We consider each part $A \in \cA$ separately and apply Lemma~\ref{lem.new1}: we see that there is a partition $\cB_A$ of $A$ such that
\[\maxin_G(\cB_A) \leq \max\{\tfrac12 e_G(A), e_0\} \leq \max \{ \tfrac12 \maxin_G(\cA), e_0 \}\]
and
\begin{equation}\label{eq.edges_between} \partial_G(\cB_A) \leq  \rho \, e_G(A) \,.\end{equation} 
Putting together the partitions $\cB_A$ gives a refinement $\cB$ of $\cA$ such that 
\[\maxin_G(\cB) \leq \max \{ \tfrac12 \maxin_G(\cA), e_0 \}\]
and
\begin{eqnarray*}
\partial_G(\cB) &=&
\partial_G(\cA) + \sum_{A \in \cA} \partial_G(\cB_A)
\; \leq \; 
\partial_G(\cA) + \rho \sum_{A \in \cA} e_G(A)\\
& \leq & \partial_G(\cA) + \rho \, (e(G) - \partial_G(\cA)) \; = \; 
(1-\rho) \,\partial_G(\cA) + \rho \, e(G).
\end{eqnarray*}
For any $B\in \cB$ and $A \in \cA$ such that $B \in \cB_A$ observe that 
\[e_G(B, \bar{B}) = e_G(B,A\backslash B) + e_G(B, \bar{A}) \leq \partial_G(\cB_A)+  e_G(A, \bar{A} ).\]
and so (c) follows by \eqref{eq.edges_between} which completes the proof.\end{proof}

We may now use Lemma~\ref{lem.new2} to complete the proof of Proposition~\ref{prop.noexptoqlb}.
\begin{proof}[Proof of Proposition~\ref{prop.noexptoqlb}]
Let $k=\lceil \log_2 \tfrac1{\alpha} \rceil$ and $e_0=\alpha e(G)$. 
Start with the trivial one-part partition~$\cA^0$ of $V(G)$, with $\maxin_G(\cA^0)=e(G)$, $\maxout_G(\cA^0)=0$  and $\partial_G(\cA^0)=0$.
Now apply the last lemma $k$ times, repeatedly refining the partition.  We obtain a sequence $\AA^1, \ldots, \AA^k$ of partitions of~$V(G)$ such that for each $j=1,\ldots,k$ 
\begin{eqnarray}
\label{eq.track_cross_edges} && \partial_G(\AA^j) \leq (1-\rho) \, \partial_G(\AA^{j-1}) + \rho \, e(G)\,,\\    
\notag && \maxin_G(\cA^j) \leq \max \{ 2^{-j}, \alpha \}\, e(G)\,, \mbox{ }\\
\notag && {\maxout_G(\cA^j) \leq \maxout_G(\cA^{j-1})+\rho \maxin_G(\cA^{j-1})\,.}
\end{eqnarray}
Denote $\cA^k$ by $\cB$. Since $2^{-k} \leq \alpha$ we have 
\begin{equation}\label{eq.maxin}
    \maxin_G(\cB) \leq \alpha\, e(G).
\end{equation}
Also
\[ \partial_G(\cB)  \leq 
 \rho\, e(G) \left( (1-\rho)^{k-1}  + \cdots + (1-\rho) + 1 \right)  =  e(G)\, (1-(1-\rho)^k)\,,\]
and
%\[ 
$1-(1-\rho)^k \leq k \rho$, 
so $\partial_G(\cB)  < k \rho \, e(G)$.
Hence $q^E_{\cB}(G) = 1- \partial_G(\cB)/e(G) > 1- k\rho \,$.\\

We may also bound $\maxout_G(\cB)$ similarly, 
\begin{equation}\label{eq.maxout}
    \maxout_G(\cB)  \leq 
 \rho\, e(G) \left( 1 + \tfrac{1}{2}+\tfrac{1}{4}+\ldots+\tfrac{1}{2^{k-1}}\right)  < 2 \rho\, e(G) \,,
\end{equation}
and hence by \eqref{eq.maxin} and \eqref{eq.maxout} (and since $\vol_G(A)\leq 2e(G)$)
\begin{eqnarray}
\notag  \max \{\vol_G(A): A \in \cB\} 
&=&
  \max \{2e_G(A)+ e_G(A,\bar{A}) : A \in \cB\}\\
\label{eq.max_vol} & \leq &
  2 \maxin_G(\cB) + \maxout_G(\cB) \;\; < \; 2\min\{1,(\alpha+\rho)\} e(G). 
\end{eqnarray}
Recall that by convexity of the sum of squares if we have non-negative $\{x_i\}_i$ with $\sum_i x_i=1$ and $\max_i x_i \leq x$ then the quantity $\sum_i x_i^2$ is maximized when we take as many $x_i=x$ as possible plus one $x_j=1-x\lfloor \frac{1}{x}\rfloor$, hence we always have $\sum_i x_i^2\leq f(x)$.\\

For $A\in \cB$ set $x_A=\vol_G(A)/2e(G)$ and thus \eqref{eq.max_vol} gives us 
\[q^D_{\cB}(G) = \sum_{A\in \cB} x_A^2 < f(\min\{1,\alpha +\rho\}) \,.\] 
Finally we have
\[ \q(G) \geq q_{\cB}(G) > 1-f(\min\{1,\alpha +\rho\}) - k\rho\,,\]
and since $\rho<\frac 3 2 \delta$ and $f$ is increasing, we obtain
\[\q(G)>1-f\left(\min\{1,\alpha+\tfrac 3 2 \delta\}\right) -  \tfrac {3k} 2 \delta\]as required. \end{proof}

\needspace{5\baselineskip}
\subsection{Proof of Theorem~\ref{thm:reslimit}}
\begin{figure}
    \centering
    \includegraphics[scale=0.7]{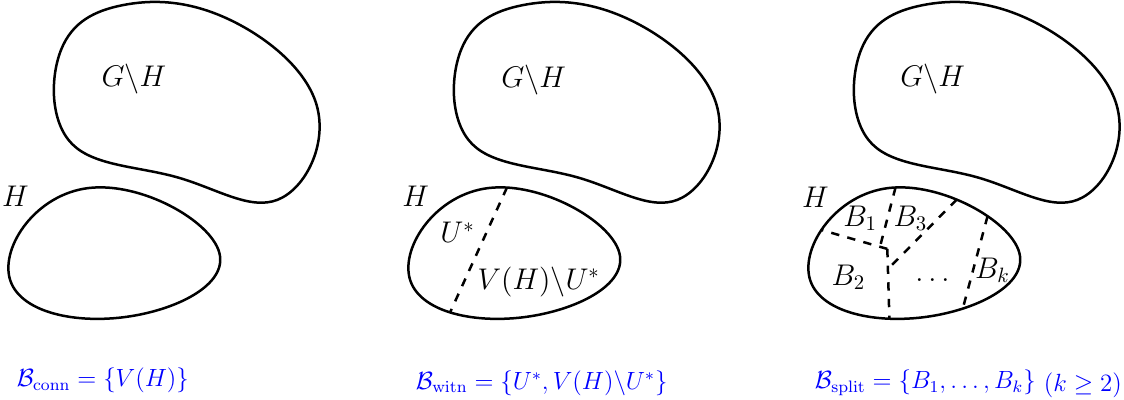}
    \caption{The three vertex partitions of $H$ considered  in the proof of Theorem~\ref{thm:reslimit}.}\label{fig.partitionH}
\end{figure}
\begin{proof}[Proof of Theorem~\ref{thm:reslimit}]
Let $\cA$ be an optimal partition of $G$. Since $G$ has no isolated vertices and $\cA$ is optimal: for any $A \in \cA$ the induced subgraph $G[A]$ must be connected. Hence any part in $\cA$ is entirely within $V(H)$ or $V(G)\backslash V(H)$. Let $\cB \subset \cA$ be the set of parts in $V(H)$. The contribution of~$H$ to the modularity score of $\cA$ is
\begin{equation}\label{eq.partialmod}
    \sum_{B \in \cB} \frac{e(B)}{e(G)} - \frac{\vol(B)^2}{\vol(G)^2}
\end{equation}
which we will denote by $q_\cB(H,m)$. We shall consider three partitions on $V(H)$: the connected component partition $\cB_{\rm conn}$ placing all of $V(H)$ into one part, $\cB_{\rm witn}$ a bipartition %of $V(H)$ 
into $\{U^*, V(H)\backslash U^*\}$ where $U^*$ is any witness, i.e. such that $\hh_H=\hh_H(U^*)$, and $\cB_{\rm split}$ a partition into at least two parts which achieves the maximal value for \eqref{eq.partialmod} i.e. $\cB_{\rm split}= \arg \max_{|\cB|>1} q_{\cB}(H,m)$. See Figure~\ref{fig.partitionH}. Note it is possible that $\cB_{\rm split}=\cB_{\rm witn}$. We proceed via the following three claims which together imply the theorem. Write $\hdelta=\hh_H$.

\begin{claim}\label{claim.res1} If $\hdelta> \alpha$, then $q_{\cB_{\rm conn}}(H, m) > q_{\cB_{\rm split}}(H, m)$.\end{claim}

\begin{claim}\label{claim.res2} If $\hdelta < \alpha$, then $q_{\cB_{\rm witn}}(H, m) > q_{\cB_{\rm conn}}(H, m)$.\end{claim}

\begin{claim}\label{claim.res3} If $\hdelta = \alpha$, then $q_{\cB_{\rm conn}}(H, m) = q_{\cB_{\rm witn}}(H, m) = q_{\cB_{\rm split}}(H, m)$.\end{claim} 

\begin{proof}[Proof of Claim~\ref{claim.res1}]
Observe first that
\begin{equation}\label{eq.Bconn}
    q_{\cB_{\rm conn}}(H, m) = \frac{e(H)}{e(G)} - \frac{\vol(H)^2}{\vol(G)^2} = \alpha - \alpha^2. 
\end{equation}
Denote by $x_B$ the proportion of the volume of $H$ contained in part $B$, $x_B=\vol(B)/\vol(H)$. Then for any partition $\cB$ of $V(H)$
\begin{equation*}
    q_\cB(H, m) = \frac{e(H)}{e(G)} - \sum_{B \in \cB} \big( \frac{e(B, V(H)\backslash B)}{\vol(G)} + \frac{\vol(B)^2}{\vol(G)^2}\big) = \alpha - \sum_{B \in \cB} \big(\frac{\alpha \; e(B, V(H)\backslash B)}{\vol(H)} + \alpha^2 x_B^2 \big)\,.
\end{equation*}

But $e(B, V(H)\backslash B)\geq \hdelta \vol(B) \vol(V(H)\backslash B)/\vol(H)$ and so by the above equation 
\begin{equation*}
    q_{\cB_{\rm split} } (H, m) \leq  \alpha -  \alpha \hdelta \sum_{B \in \cB_{\rm split}} x_B(1-x_B) - \alpha^2 \sum_{B \in \cB_{\rm split} } x_B^2 = \alpha - \alpha \hdelta + \alpha(\hdelta - \alpha) \sum_{B \in \cB_{\rm split} } x_B^2
\end{equation*}
where the equality above is by noting $\sum_{B \in \cB_{\rm split} } x_B = 1$. Hence by~(\ref{eq.Bconn})
\begin{equation}\label{eq.Bconn_Bsplit}
q_{\cB_{\rm conn}}(H, m) - q_{\cB_{\rm split}}(H, m) \geq \alpha(\hdelta - \alpha)(1-\textstyle{\sum_{B \in \cB_{\rm split}}} x_B^2).
\end{equation}
Since $\sum_{B \in \cB_{\rm split} } x_B^2< 1$, if $\hdelta> \alpha$ we have $q_{\cB_{\rm conn}}(H, m) > q_{\cB_{\rm split}}(H, m)$ as required.\end{proof}

\begin{proof}[Proof of Claim~\ref{claim.res2}]
Note that since $U^*$ a witness of $\hat{h}_H=\hdelta$, we get that $e(U^*,V(H)\backslash U^*) = \hdelta \vol(U^*) \vol(V(H)\backslash U^*) / \vol(H)$. Thus similarly to $q_{\cB_{\rm split} } (H, m)$,
\begin{equation*}
    q_{\cB_{\rm witn}}(H, m) = \alpha - \alpha \hdelta + \alpha(\hdelta - \alpha) (x_{U^*}^2 + (1-x_{U^*})^2) ,
\end{equation*}
and so
\begin{equation}\label{eq.Bwitn_Bconn}
q_{\cB_{\rm witn}}(H, m) - q_{\cB_{\rm conn}}(H, m) \geq \alpha(\alpha - \hdelta)(1-x_{U^*}^2 -(1-x_{U^*})^2 ) .
\end{equation}
Hence if $\alpha > \hdelta$ we have $q_{\cB_{\rm witn}}(H, m) > q_{\cB_{\rm conn}}(H, m)$ as required. \end{proof}
\begin{proof}[Proof of Claim~\ref{claim.res3}]
This follows almost immediately from previous calculations. By \eqref{eq.Bconn} and~\eqref{eq.Bwitn_Bconn}, if $\hdelta=\alpha$ then $q_{\cB_{\rm conn}}(H, m) = q_{\cB_{\rm witn}}(H, m)$; and by \eqref{eq.Bconn_Bsplit} if $\hdelta=\alpha$ then $q_{\cB_{\rm conn}}(H, m) \geq q_{\cB_{\rm split}}(H, m)$. However by definition, always $q_{\cB_{\rm split}}(H, m) \geq q_{\cB_{\rm witn}}(H, m)$; and hence we must have equality between all three, as required.\end{proof}\end{proof}

\needspace{5\baselineskip}
\subsection{Proof of Proposition~\ref{prop.q0}}\label{subsec.propq0}
In this subsection we often drop the subscript $G$ as it is not needed.
\begin{proof}[Proof of Proposition~\ref{prop.q0}]
We first show (a) and (c) to be equivalent. \\

Let $G=(V,E)$ be a nonempty graph, and let $\nu = \vol(G)$.  For $U \subseteq V$, let $p_G(U) = 2e(U) \nu - \vol(U)^2$.  Observe that $p_G(\emptyset) = p_G(V) =0$.
If $\cA$ is a partition of $V(G)$, then $q_{\cA}(G) = \nu^{-2}  \sum_{A \in \cA} p_G(A)$.  Since $\vol(U) = 2e(U) + e(U,\bar{U})$ (where $\bar{U}$ denotes $V \backslash U$), we have
\begin{eqnarray*}
  p_G(U) 
        & = & 2e(U) \nu - \vol(U)^2 = (\vol(U) - e(U,\bar{U})) \nu - \vol(U) (\nu - \vol(\bar{U}))\\
        & = & - e(U,\bar{U}) \nu  + \vol(U) \vol(\bar{U}).
\end{eqnarray*}

Thus we may also write $p_G(U)$ as $\vol(U) \vol(\bar{U}) - e(U, \bar{U}) \nu$; and this expression for $p_G(U)$ is symmetric in $U$ and $\bar{U}$, so $p_G(U) = p_G(\bar{U})$.
Now, if $\cA$ is a bipartition with parts $U$ and $\bar{U}$, then
\[ q_{\cA}(G) = \nu^{-2} (p_G(U) + p_G(\bar{U})) = 2 \nu^{-2} p_G(U).\]
Thus there is a bipartition $\cA$ with $q_{\cA}(G)>0$ iff there is a set $U$ of vertices with $p_G(U)>0$, iff $\q(G)>0$. Equivalently, $\q(G)=0$ iff $p_G(A) \leq 0$ for each $A \subset V$, iff
\[ e(A, \bar{A}) \geq  \frac{\vol(A)\, \vol(\bar{A})}{\vol(G)} \;\; \mbox{ for each } A \subseteq V\,. \]
This gives the equivalence of (a) and (c) in Proposition~\ref{prop.q0}.\\

Now let us prove the equivalence of (a) and (b) in Proposition~\ref{prop.q0}.
Let $\emptyset \neq A \subset V$ and let $\cA$ be the bipartition of $V$ with parts $A$ and $\bar{A}$. Then
\begin{eqnarray*}
m^2 q_{\cA}(G)
  &=&
 (e(A)+e(\bar{A}))\, m - \tfrac14 ((2e(A)+e(A,\bar{A}))^2 + (2e(\bar{A})+e(A,\bar{A}))^2\\
 &=& 
 (e(A)+e(\bar{A})) \left(e(A)+e(\bar{A}) + e(A, \bar{A}) \right) \\
 && - \left(e(A)^2 + e(\bar{A})^2 +  (e(A)+e(\bar{A}))\, e(A,\bar{A}) + \tfrac12 e(A, \bar{A})^2 \right) \\ 
 &=& 
 (e(A)+e(\bar{A}))^2 - \left(e(A)^2 + e(\bar{A})^2 + \tfrac12 e(A, \bar{A})^2 \right) \\
 & = &
 2 e(A) e(\bar{A}) - \tfrac12 e(A, \bar{A})^2\,.
\end{eqnarray*}
Hence
\[ q_{\cA}(G)>0 \; \mbox{ iff }\; e(A, \bar{A}) < 2 \sqrt{e(A) e(\bar{A})}\,,\] 
and so (since $\q(G)>0$ iff $q_\cA(G)>0$ for some bipartition $\cA$)
\[ \q(G)>0 \; \mbox{ iff }\; e(A, \bar{A}) < 2 \sqrt{e(A) e(\bar{A})} \;\; \mbox{ for some } A \subseteq V\,.\]
This completes the proof.\end{proof}

\needspace{8\baselineskip}
\section{Constructions}\label{sec.constructions}
\subsection{Constructions related to Theorem~\ref{thm}}\label{sec.example_thm1_1}

The following proposition gives example constructions for $G_\alpha$ and $G_H$ in the `converse' statements (a)(ii) and (b)(ii) of Theorem~\ref{thm}.  
\needspace{5\baselineskip}
\begin{proposition}\label{prop.example} Let $f(\alpha)=\alpha^2\lfloor1/\alpha\rfloor + (1-\alpha \lfloor1/\alpha \rfloor)^2$. For all $0< \alpha <1$ and for $m$ large enough,
    \begin{enumerate}[label=(\alph*)] 
        \item For any graph $H$ with $e(H) \leq \alpha m$, let $G_H$ on $m$ edges be constructed by adding  $m-e(H)$ disjoint edges. Then $H$ is an induced subgraph of $G_H$ and $\q(G_H)\geq 1-\alpha^2-O(1/m)$.
        \item Let $G_{\alpha}$ on $m$ edges be constructed by taking as many cliques on $\lfloor \sqrt{2\alpha m} \rfloor$ vertices as possible, as large as possible a clique on the remaining edges, and disjoint edges for the rest. Then $\q(G_\alpha) = 1-f(\alpha)+O(m^{-1/2})$ and any subgraph $H$ of $G_\alpha$ with $e(H)>\alpha m$ is disconnected.
    \end{enumerate}
\end{proposition}

\begin{proof}[Proof of (a)]\label{proof.GH}
For the graph $G_H$ in part (b), let $\cA$ be the partition with one part $V(H)$ and one part for each added isolated edge.  Then
\[ q^* (G_H) \geq q_\cA(G_H) = 1 - (\vol(H)^2+ 4(m - e(H))/(4m^2) \geq 1 - \alpha^2 +O(1/m),\]
as required.\end{proof}

\begin{proof}[Proof of (b)]
\label{proof.Galpha}
By Corollary~\ref{cor.0modpart}, since complete graphs have modularity value 0, the modularity optimal partition is the connected components partition $\cC$, where we place each disjoint clique (including each disjoint edge) into a separate part. Clearly $q_\cC(G_{\alpha})=1-q^D_\cC(G_{\alpha})$. To calculate the degree tax requires bounds on the number of edges in each clique.
Let $x$ be the number of edges in a clique on $\lfloor  \sqrt{2\alpha m} \rfloor$ vertices.  Then $x < \alpha m$ and
\[ x > \tfrac12 ( \sqrt{2 \alpha m}-1)( \sqrt{2 \alpha m}-2)= \alpha m (1+O(m^{-1/2})). \]
Hence there are $\lfloor \frac{1}{\alpha} \rfloor$ such cliques for $m$ sufficiently large, with a total of  $\lfloor \frac{1}{\alpha} \rfloor \alpha m (1+O(m^{-1/2}))$ edges. Thus the remaining clique has $(1-\alpha\lfloor \frac{1}{\alpha} \rfloor) m (1+O(m^{-1/2}))$ edges.
The degree tax is at least the contribution of these $\lfloor \frac{1}{\alpha} \rfloor+1$ cliques, (i.e. not including the disjoint edges), thus %and thus
\[
q^D_\cC(G_\alpha) \geq \lfloor \frac{1}{\alpha} \rfloor\cdot \alpha^2 (1+O(m^{-1/2}))^2 + (1-\alpha\lfloor \frac{1}{\alpha} \rfloor)^2(1+O(m^{-1/2}))^2 = f(\alpha) + O(m^{-1/2})
\]
as required.\end{proof}

%%%%%%%%%%%%%%%%%%%%%%%%%%%%%%%%%%%%%%%%
\needspace{5\baselineskip}
\subsection{Constructions related to Proposition~\ref{prop.exptoqub}(ii)} \label{sec.constr-prop1.2(ii)}
\begin{figure}
\centering
\begin{subfigure}[t]{0.45\textwidth}
\centering
\includegraphics[scale=0.5]{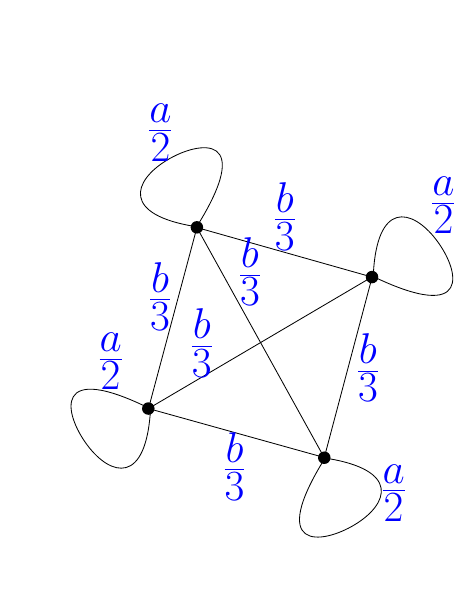}
\caption{The weighted graph constructed by taking the complete graph on four vertices, $K_4$, and adding a loop to each vertex. Note that each weighted degree is~$a + b$.
}\label{subfig.completeloop}
\end{subfigure}%
\hspace{8mm}
~
\begin{subfigure}[t]{0.45\textwidth}
\centering
\includegraphics[scale=0.5]{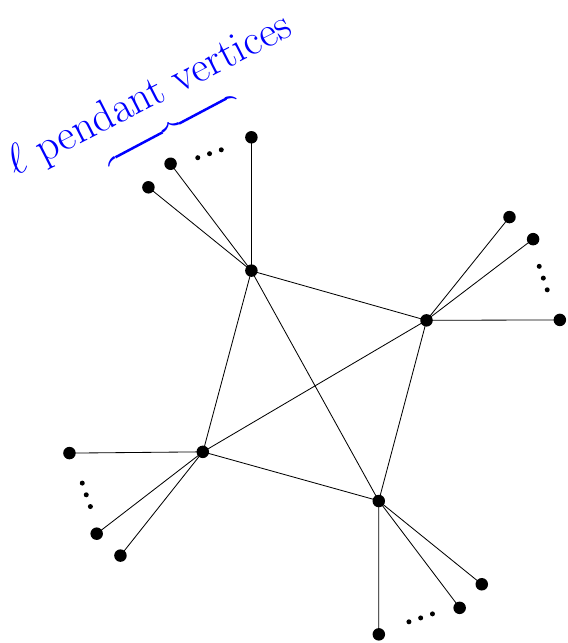}
\caption{The graph constructed by taking the complete graph on four vertices, $K_4$, and adding~$\ell$ pendant vertices to each of its vertices.
}\label{subfig.completehairy}
\end{subfigure}
\caption{Example construction for Proposition~\ref{prop.exptoqub}(ii) in (b). In Lemma~\ref{lem.Kleaves} we show the modularity value and expansion constant $\hh$ of the simple graph in (b) is the same as for the weighted graph with loops in (a) for the correct parameter values.}\label{fig.completeloopANDhairy}
\end{figure}
In this section we will construct a family of graphs and establish their modularity values and the expansion $\hh_H$ for a subgraph $H$. This will prove Proposition~\ref{prop.exptoqub}(ii), and thus show that the inequality in Proposition~\ref{prop.exptoqub}(i) is tight.\\

The constructed graph $G$ will consist of disjoint edges together with the subgraph $H$, where~$H$ is a complete graph with pendant vertices (leaves) - see Figure~\ref{fig.completeloopANDhairy}(b). Loosely, in $H$ the pendant vertices are always grouped together with the adjacent vertex, both for minimal expansion vertex subsets (see Lemma~\ref{lem.hdeltabothconn}) and for maximal modularity partitions (see Lemma~\ref{lem.modpartsconn}).
This will mean that it is enough to establish expansion properties for~$H'$ a weighted complete graph with loops, and modularity values of $G'$ where~$H$ is replaced by~$H'$.\\

We first show that we may construct a weighted graph with loops $G_w$ fulfilling the claims above, and then show that we may approximate this graph by a simple graph. See Figure~\ref{fig.completeloopANDhairy}(a) for an illustration of $H_w$ which is used to construct $G_w$.

\needspace{5\baselineskip}
\begin{lemma}\label{lem.Kloops} 
For $k \geq 2$ and $a,b>0$, let $H_w=H_w(a,b,k)$ be the weighted graph with loops constructed by taking the complete graph on $k$ vertices with all $\binom{k}{2}$ edges of weight $b/(k-1)$ and adding  a loop of weight $a/2$ at each vertex. Given also $0 < \alpha < 1$, construct $G_w=G_w(H_w, \alpha)$ by taking the disjoint union of~$H_w$ and $\lfloor e(H_w)(1-\alpha)/\alpha \rfloor$ disjoint edges each of weight~1. Then 
\begin{enumerate}[label=(\roman*)]  
    \item %for $k\geq 2$,
    \[\hh_{H_w} = \frac{b}{a+b}(1+\frac{1}{k-1}) ; \] 
    \item for any $\eps>0$ we may take $k$ large enough such that \[ \q(G_w) \geq 1 - \alpha\min\{\hh_{H_w} , \alpha\} - \eps.\]
\end{enumerate}
\end{lemma}

Observe that the graph $H_w$ is regular with weighted degree $a+b$ at each vertex, and thus $\vol(H_w)=k(a+b)$.

\begin{proof}[Proof of (i)]\label{proof.Kloops} 
Let $U_j$ be a set of $1\leq j<k$ vertices of $H_w$ (by symmetry all such sets of size~$j$ are equivalent), and note that
\begin{equation}\label{eq.allsetsminimisehh}
\hh_{H_w}(U_j) =  \frac{e(U_j, \bar{U_j}) \vol(H_w)}{\vol(U_j)\vol(\bar{U_j})} = \frac{j(k-j) b (k-1)^{-1} \cdot k (a+b)}{j (a+b) \cdot (k-j) (a+b) }%= \frac{bk}{(k-1)(a+b)}
= \frac{b}{a+b}\big(1+\frac{1}{k-1}\big)
\end{equation}
which is independent of $j$. But then
\[ \hh_{H_w} = \min_j \hh_{H_w}(U_j) = \frac{b}{a+b}\big(1+\frac{1}{k-1}\big)\]
as required.\end{proof}

\begin{proof}[Proof of (ii)]
Let $m=e(G_w)$ and let $\alpha' = e(H_w)/e(G_w)$. Note that $m =e(H_w)+\lfloor e(H_w)(\alpha^{-1}-1) \rfloor$ so $\alpha \leq \alpha' < \alpha(1+1/m)$. The modularity optimal partitions of $G_w$ will place each disjoint edge in its own part and hence 
\begin{eqnarray}
 \q(G_w) 
 & = & (1 - \alpha')(1-1/m) + \max_{\cB} \sum_{B \in \cB}  \frac{\vol(B)}{\vol(G_w)} \Big(\frac{2e(B)}{\vol(B)}-\frac{\vol(B)}{\vol(G_w)} \Big) \notag \\
 & = & (1 - \alpha')(1-1/m) + \alpha' \max_{\cB} \sum_{B \in \cB} \frac{\vol(B)}{\vol(H_w)}  f(B) \label{eq.whichB}
\end{eqnarray}
where the max is over partitions $\cB$ of $V(H_w)$ and \[f(B)= \frac{2e(B)}{\vol(B)}-\frac{\vol(B)}{\vol(G_w)}.\]
Again, let $U_j$ be a set of $j$ vertices of $H_{w}$.  We may calculate that $e(U_j) = (\binom{j}{2}b/(k-1) + ja/2)$ and $\vol(U_j)=j(a+b)$. Also, that $2m = \vol(H_w)/\alpha' = k (a+b)/\alpha'$. Hence
\[
f(U_j) = \frac{j(j-1)b/(k-1) + ja}{j(a+b)} - \frac{\alpha' j(a+b)}{k(a+b)} = \frac{(j-1)b/(k-1)  +  a}{a+b} - \frac{\alpha' j}{k}.
\]
Rearranging, and recalling the expression for $\hh_{H_w}$ shown in part (i), we get
\[
f(U_j) = \frac{j}{k} \Big( %\underbrace{ 
\frac{bk}{(k-1)(a+b)}  - \alpha'  %}_{*} 
\Big) +  \frac{a}{a+b} -  \frac{b}{(k-1)(a+b)} = \frac{j}{k}(\hh_{H_w}-\alpha') + 1 - \hh_{H_w},
\]
where the second equality follows by adding $b/(a+b)$ to the second term and subtracting $b/(a+b)$ from the third term. Thus the $j$ which maximises $f(U_j)$ depends on the relative values of $\hh_{H_w}$ and~$\alpha'$ :
\[ \max_j f(U_j) = 
    \begin{cases} 
        f(U_1)=1 - \hh_{H_w} + \tfrac{1}{k}(\hh_{H_w}-\alpha') &\mbox{ if $\hh_{H_w} < \alpha'$ } \\  
        f(U_k)=1 - \alpha' &\mbox{ if $\hh_{H_w} \geq \alpha'$ .}
    \end{cases} 
    \]
Since the sum in~\eqref{eq.whichB}
 is a weighted average of values $f(B)$, its maximum value is at most $\max_j f(U_j)$, and clearly $f(U_1)$ and $f(U_m)$ can be achieved.
Hence
  \[\max_\cB \sum_{B \in \cB}  \frac{\vol(B)}{\vol(H_w)} f(B) =
  \begin{cases} 
     1- \hh_{H_w} + \frac1{k}(\hh_{H_w} - \alpha') &\mbox{ if $\hh_{H_w} < \alpha'$ } \\   
      1-\alpha' &\mbox{ if $\hh_{H_w} \geq  \alpha'$ }
  \end{cases} \]
and so by~(\ref{eq.whichB})
\[ \q(G_w) = 
 \begin{cases} 
     (1 - \alpha')(1-1/m)+\alpha' \big(1- \hh_{H_w} + \frac1{k}(\hh_{H_w} - \alpha')\big) &\mbox{ if $\hh_{H_w} < \alpha'$ } \\   
     (1 - \alpha')(1-1/m) + \alpha' ( 1-\alpha') &\mbox{ if $\hh_{H_w} \geq \alpha'$ .}
  \end{cases} \]
By taking $k$ large enough (which for fixed $a,b$ implies $m$ large enough) 
\[ \q(G_w) \geq 1 - \alpha\min\{\hh_{H_w}, \alpha\} - \eps\]
as required.\end{proof}

\needspace{5\baselineskip}
\begin{lemma}\label{lem.Kleaves}
For $k \geq 2$ and $\ell \geq 1$, let $H=H(k,\ell)$ be the simple graph constructed by taking the complete graph on $k$ vertices with and adding $\ell$ leaves at each vertex. Construct $G=G(H, \alpha)$ by taking the disjoint union of $H$ and $\lfloor e(H)(1-\alpha)/\alpha %(\alpha^{-1}-1) 
\rfloor$ disjoint edges. Let $H_w=H_w(2\ell, k-1 ,k)$ and $G_w=G_w(H_w, \alpha)$ be the weighted graphs with loops as defined in Lemma~\ref{lem.Kloops}. Then 
    \begin{enumerate}[label=(\roman*)] 
        \item $\q(G) = \q(G_w)$ and
        \item %if $\ell \geq 1$, 
        $\hh_{H} = \hh_{H_w}$. 
    \end{enumerate}
\end{lemma}

To prove Lemma~\ref{lem.Kleaves} we need results on which vertex subsets $U$ minimise $\hh(U)$ in $G$ and~$G_w$ and which vertex partitions $\cA$ maximise $q_\cA(G)$ and $q_\cA(G_w)$. %todo We pause to collect these and then provide the proof of the lemma.
The following result is elementary %, and can be found as an exercise in ?, 
but for completeness we provide a proof.
\needspace{5\baselineskip}
\begin{lemma}\label{lem.hdeltabothconn}
Let $H$ be a connected graph and let $\emptyset \neq U^* \subsetneq V(H)$ be such that $\hh_H=\hh(U^*)$. Then $H[U^*]$ and $H[V \backslash U^*]$ are connected.
\end{lemma}

\begin{proof}
Suppose not and let $U^*$ be such that $\hh_H=\hh_H(U^*)$ but  at least one of $H[U^*]$ and $H[V \backslash U^*]$ is disconnected. Choose $U \in \{U^*, V \backslash U^*\}$ with $G[U]$ disconnected. We will show there exists $\emptyset \neq U' \subsetneq U$ such that $\hh(U') < \hh(U) = \hh(U^*)$ which will prove the result.\\

As $H[U]$ is disconnected, it is the disjoint union of $H[U_1]$ and $H[U_2]$ for some $\emptyset \neq U_1\subsetneq U$ and~$U_2=U \backslash U_1$. Note that since $e(U_1, U_2)=0$ we have
\[ e(U, V \backslash U) = e(U_1, V \backslash U_1) + e(U_2, V \backslash U_2)
\]
and thus
\[
\frac{\hh(U)}{\vol(H)}%=\frac{e(U, V\backslash U)}{\vol(U)\vol(V\backslash U)} 
= \frac{\vol(U_1)}{\vol(U_1)\!+\!\vol(U_2)}\underbrace{ \frac{e(U_1, V\backslash U_1)}{\vol(U_1)\vol(V\backslash U)} }_{\lozenge} + \frac{\vol(U_2)}{\vol(U_1)\!+\!\vol(U_2)}\underbrace{  \frac{e(U_2, V\backslash U_2)}{\vol(U_2)\vol(V\backslash U)}}_{\blacklozenge}.
\]

Since $\hh(U)/\vol(H)$ is a weighted average of terms ($\lozenge$) and ($\blacklozenge$), one of the terms must be at most~$\hh(U)/\vol(H)$: w.l.o.g.\ this is ($\lozenge$) the term involving $U_1$. Hence we may assume
\begin{equation}\label{eq.wlogu1small}
\frac{\hh(U)}{\vol(H)} \geq \frac{e(U_1, V\backslash U_1)}{\vol(U_1)\vol(V\backslash U)}.
\end{equation}

However, $H$ has no isolated vertices so $\vol(V\backslash U)  <  \vol(V\backslash U_1)$,
and thus in particular
\[
\frac{\hh(U_1)}{\vol(H)} 
= \frac{e(U_1, V\backslash U_1)}{\vol(U_1)\vol(V\backslash U_1)} < \frac{e(U_1, V\backslash U_1)}{\vol(U_1)\vol(V\backslash U)} .
\]
Hence by~\eqref{eq.wlogu1small} we have $\hh(U_1)< \hh(U)$, as required.\end{proof}

The following lemma will imply that for the graph $G$ in Lemma~\ref{lem.Kleaves} - see also Figure~\ref{fig.completeloopANDhairy}(b) - for any modularity optimal partition each pendant leaf is always together in the same part as the vertex in the complete graph which they are incident with.
%
%
%\needspace{2\baselineskip}
\begin{lemma}[\cite{nphard,w1hard}]\label{lem.modpartsconn}
Suppose that $G$ is a graph that contains no isolated vertices and no loops.
If $\cA$ is a partition of $V(G)$ 
such that $q_\cA(G) = \q(G)$ then, for every $A \in \cA$, $G[A]$ is
a connected subgraph of $G$ with at least one edge.
\end{lemma}

We are now ready to prove that the graph illustrated in Figure~\ref{fig.completeloopANDhairy}(b) has the modularity claimed.

\begin{proof}[Proof of Lemma~\ref{lem.Kleaves}(i)]\label{proof.Kleaves}
Given a graph $J$ let $\widetilde{J}$ be the graph obtained as follows: for each vertex~$v$ with degree at least 2 and with $\ell \geq 1$ adjacent leaves, delete the leaves and add a loop at $v$ with weight $\ell$; and for each isolated edge $uv$ delete the edge and add a loop of weight 1 at $u$ (say).  Each vertex partition $\cA$ of $J$ induces a partition $\widetilde{\cA}$ of $\widetilde{J}$, where~$\widetilde{\cA}$ is obtained from $\cA$ by deleting vertices in $V(J)\backslash V(\widetilde{J})$.\\

Recall that the graph $G=G(H,\alpha)$ consists of disjoint edges together with the graph $H(k, \ell)$ (the $k$-clique with $\ell$ pendant leaves at each of its $k$ vertices). Thus $\widetilde{G}$ consists of disjoint loops and a $k$-clique with a weight $\ell$ loop at each vertex and weight 1 edges in the clique, so $\widetilde{G}$ is the graph $G_w$.\\ 

We say that a partition $\cA$ for a graph $J$ is \emph{pendant-consistent} if for each pendant vertex $v$ with pendant edge $uv$ there is a part $A \in \cA$ with $u,v \in A$.  It is easy to see that if $\cA$ is pendant consistent for $J$ then $q_\cA(J)=q_{\widetilde{\cA}}(\widetilde{J})$.  
But by Lemma~\ref{lem.modpartsconn}, any modularity optimal partition for $J$ is pendant consistent, so $\q(J)=\q(\widetilde{J})$.  Hence $\q(G)=\q(\widetilde{G})= \q(G_w)$, as required.\end{proof}

\begin{proof}[Proof of Lemma~\ref{lem.Kleaves}(ii)]\label{proof.Kleaves2}
We first observe that for any set $\{u\}$ consisting of a single vertex $u \in V(H)$ we have $\hh_H(\{u\})>1$. To see this, note that since $H$ has no loops, $e(\{u\}, V(H)\backslash \{u\}) = \deg(u) = \vol(\{u\})$, and hence \[ \hh_H(\{ u \}) = \frac{\vol(H)}{\vol(V(H)\backslash \{u\})} = \frac{\vol(H)}{\vol(H)-\deg(u)}>1
\] as by construction $H$ has no isolated vertices.\\

Also, note that $\hh_H <1$. To see this, for example let $U$ consist of $1 \leq j \leq k-1$ vertices from the $k$-clique together with their $j \ell$ adjacent leaves: then
\[
\hh_H(U) = \frac{j (k-j) \cdot k (2\ell+k-1)}{j (2\ell+k-1) \cdot (k-j) (2\ell+k-1)} = \frac{k}{2\ell+k-1} <1.
\]

Let $U_2$ minimise $\hh_H(U)$ over $\emptyset \neq U \subsetneq V(H)$; and note that by the above we must have $2 \leq |U_2| \leq n-2$, where $n=|V(H)|=k(\ell+1)$.

\begin{claim}\label{claim.atleasttwo} 
The partition $\{ U_2, V(H)\backslash U_2\}$ is pendant-consistent.\end{claim}

\begin{proof}[Proof of Claim~\ref{claim.atleasttwo}] Suppose for a contradiction that there exists a leaf $u$ and pendant edge $uv$ such that $u \in U_2$ and $v\in V(H)\backslash U_2$ or vice-versa. W.l.o.g. we may assume $u \in U_2$ and $v\in V(H)\backslash U_2$. Recall by Lemma~\ref{lem.hdeltabothconn} both $H[U_2]$ and $H[V(H)\backslash U_2]$ are connected. Since $|U_2| \geq 2$ there must be another vertex $w$ with $w\in U_2$. However, since $v \notin U_2$ and $u$ is a leaf it must be that $H[U_2]$ is disconnected and we have proven the claim. %our contradiction.
\end{proof}

Similar to the proof of part (i), for $U$ such that the partition $\{U, V(H)\backslash U \}$ is pendant-consistent we have $\hh_H(U)=\hh_{H_w}(\widetilde{U})$. Thus by the claim we have \[ \hh_H(U_2) = \hh_{H_w}(\widetilde{U_2})=\hh_{H_w} \] since the minimum of $\hh_{H_w}(B)$ is obtained by all non-empty sets $B\subset V(H_w)$ - see \eqref{eq.allsetsminimisehh}. This completes the proof of part (ii).
\end{proof}

\needspace{5\baselineskip}
\subsection{Constructions related to Corollary~\ref{cor.zeromodimplieshalf}}
The following lemma shows that the converse of Corollary~\ref{cor.zeromodimplieshalf} does not hold.

\needspace{5\baselineskip}
\begin{lemma}\label{lem.windmill} Windmill graphs. For $\ell\geq 2$ construct the windmill graph $W_\ell$ by taking a star with $2\ell$ leaves and adding a perfect matching between the leaf vertices. Then 
\begin{equation}\label{eq.windmill_exp} h_{W_\ell}=1/2 \end{equation}
and
\begin{equation}\label{eq.windmill_mod} \q(W_\ell)>0. \end{equation}
\end{lemma}

See Figure~\ref{fig.windmill} for an illustration of $W_8$. %todo recall  why this was interesting in relation to corollary~\ref{cor.zeromodimplieshalf}?

\begin{figure}
\centering
\includegraphics[scale=0.6]{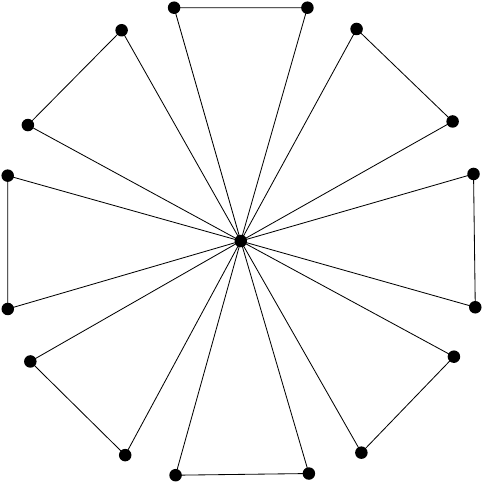}
\caption{The windmill graph $W_8$, see Lemma~\ref{lem.windmill}.}\label{fig.windmill}
\end{figure}
\begin{proof}[Proof of \eqref{eq.windmill_exp}]\label{proof.windmill}
We will show that for each non-empty set $X$ of vertices with $\vol(X)\leq \vol(\bar{X})$ we have \[h_{W_\ell}(X)=e(X,\bar{X})/\vol(X)\geq 1/2\] with equality for some such set $X$. Observe that $\vol(W_\ell)=6\ell$, and so we consider $X$ with $\vol(X)\leq 3\ell$.\\ 

Denote the central vertex by $v$. Notice that to show $h_{W_\ell}(X)\geq 1/2$ for all
$X$ as above it suffices to show that $h_{W_\ell}(X)\geq 1/2$ for $X$ with $W_\ell[X]$ connected. % as \mbox{$\min\{a/b, c/d\}\leq (a+b)/(c+d)$}. 
Hence if $v\notin X$ there are only two cases to check: if $|X|=1$ then $X$ consists of a single non-central vertex and $h_{W_\ell}(X)=e(X,\bar{X})/\vol(X)=2/2=1$, and if $|X|=2$ then $X$ consists of two non-central vertices connected by an edge and $h_{W_\ell}(X)=e(X,\bar{X})/\vol(X)=2/4=1/2$. Note this second case shows that $h_{W_\ell}\leq 1/2$.\\

Now it remains only to consider vertex subsets containing the central vertex. Let $X_j$ consist of the central vertex $v$ and $j$ non-central vertices. Then $\vol(X_j)=2\ell+2j$ and thus we may assume $j\leq \ell/2$ (so that $\vol(X_j)\leq 3\ell)$. Hence, considering just edges between the central vertex $v$ and vertices in $\bar{X_j}$, we have that $e(X_j,\bar{X_j})\geq 2\ell-j \geq 3\ell/2$. Thus, 
\[
h_{W_\ell}(X_j)=e(X_j,\bar{X}_j)/{ \vol(X_{j})} \geq (3\ell/2)/(3\ell)\geq 1/2,
\]
which completes the proof of~\eqref{eq.windmill_exp}.\end{proof}

\begin{proof}[Proof of \eqref{eq.windmill_mod}.]

Let $A$ be a vertex set consisting of an adjacent pair of non-central vertices and define the bipartition $\cA=\{A, V(W_\ell)\backslash A\}$.
This has edge contribution $q^E_\cA(W_\ell)=1-2/(3\ell)$. To calculate the degree tax, note $\vol(A)=4$, and so $\sum_{A\in\AA} \vol(A)^2 = 4^2 + (6\ell-4)^2$ and thus $q^D_\cA(W_\ell)=1 - (12\ell-8)/(9\ell^2)$. Hence for $\ell\geq 2$, 
\[\q(W_\ell)\geq q_\AA(W_\ell)=\frac{6\ell-8}{9\ell^2}>0\]
which completes the proof of~\eqref{eq.windmill_mod}.
\end{proof}

%%%%%%%%%%%%%%%%%%%%%%%%%%%%%%%%%%%%%%

\needspace{3\baselineskip}
\section{Parameterising by volume}
\label{sec.choice} 
In this paper so far we have considered the modularity of graphs $G$ with subgraphs $H$ of relative size $e(H)/e(G) \geq \alpha$. We also could have considered subgraphs $H$ of relative volume \mbox{$\vol_G(H) / \vol(G) \geq\beta$}; that is, parameterising by the relative volume of the subgraph instead of the relative number of edges. (Here we write $\vol_G(H)$ for $\vol_G(V(H))$.)
Note that if $H$ has relative size $\alpha$ then it has relative volume $\geq \alpha/2$ but the relative volume could be $> 1/2$ (even for small $\alpha >0$ - see the construction that follows).
Let us consider briefly whether there are corresponding versions of Propositions~\ref{prop.exptoqub}~(i) and~\ref{prop.noexptoqlb}~(i) under this parameterisation.\\

The lower bound for $\q(G)$ (in terms of $\beta$ and $h_H$)
analogous to that in Proposition~\ref{prop.noexptoqlb}~(i) (in terms of $\alpha$ and $h_H$) works nicely, see Proposition~\ref{prop.noexptoqlb2} below. However, there is no upper bound for $\q(G)$ corresponding to that in Propositions~\ref{prop.exptoqub}~(i), 
%is less satisfying, 
as we may see from the following example. For any $\eps >0$, we  construct a graph $G$ with modularity $\q(G) \geq 1-\eps$ and an induced subgraph $H$ of $G$ with $h_H=1$ which has volume at least half that of $G$. Recall from Proposition~\ref{prop.exptoqub}~(i) that, for all $0< \alpha \leq 1$, for every graph $G$ with subgraph $H$ of relative size $e(H)/e(G) \geq \alpha$, we have
$\q(G) \leq 1 - \alpha \min\{h_H, \alpha\}$.  Let $\eta>0$ be arbitrarily small.
The graphs $G$ which we shall construct show that it is {\bf not} true that for all $0< \beta \leq 1$, for every graph $G$ with an induced subgraph $H$ of relative volume $\vol(H)/\vol(G) \geq \beta$, we have
$\q(G) \leq 1 - \eta \, \beta \min\{h_H, \beta\}$.\\

Given $\eps>0$ we construct such a graph $G$ as follows.  Let $k \geq 2/\eps$, let $H$ be the $k$-edge star, and let~$G$ be formed by attaching $k$ leaves to each leaf of $H$. Equivalently let $G$ be a $k$-ary tree of depth two. Thus $G$ has $k^2 +k+1$ vertices and $k^2+k$ edges. Also $h_H=1$; and $\vol_G(H) =k^2 +2k$, so $\vol_G(H)/\vol(G) > 1/2$.
Let $\cA$ be the $(k+1)$-part partition of $V(G)$ with a part for each leaf of the inner star $H$ along with its adjacent leaves, and a part for the central vertex (on in its own).  Then $q_\cA^E(G)= 1-k/(k^2+k)>1-1/k$.
Also $\cA$ has $k$ parts of the same volume (namely $2k+1$) and one of smaller volume (namely $k$), so $q_\cA^D(G) < 1/k$. Thus $\q(G) \geq q_\cA(G) > 1-2/k \geq 1-\eps$,  as claimed.\\

The following result gives a lower bound on $\q(G)$ corresponding to that in Proposition~\ref{prop.noexptoqlb} (i). The function $f$ is defined in \eqref{eqn.f}.

\needspace{5\baselineskip}
\begin{proposition} \label{prop.noexptoqlb2}
Let  $0< \beta, \delta \leq 1$, let $G$ be a graph, and suppose that each induced subgraph $H$ with relative volume $\vol_G(H)/ \vol(G) > \beta$ has $h_H < \delta$.
Then 
\[\q(G) \geq 1-f(\beta) - 2 \delta\, \lceil \log_2 \tfrac1{\beta} \rceil \,. \]
\end{proposition}
The proof follows the general approach of Laso{\'n} and Sulkowska~\cite{lason2021modularity}. 
\begin{proof}
Let $V=V(G)$.
Given a graph $J$ on $V$ we let $\CC( J)$ denote the partition of $V$ into the vertex sets of the connected components of $J$.
We start with the graph $G$ and $\CC(G)$, 
and repeatedly delete edges to form a subgraph $G'$ with partition $\CC(G')$, until no large components remain.  Each component of each graph $G'$ is
an induced subgraph of $G$.
\smallskip

\hspace{.3in} {\bf while} the current graph $J$ is such that some part $A$ in $\CC(J)$ has $\vol_G(A) > \beta\, \vol(G)$

\hspace{.5in} find $\emptyset \neq S \subset A$ such that $\vol_G(S) \leq \vol_G(A \backslash S)$ and $e_G(S,A \backslash S) < \delta\, \vol_G(S)$, and

\hspace{.5in}
delete the edges in $E_G(S,A \backslash S)$ to form the next current graph. 
\smallskip

Note that in the `while loop' above, the component of $J$ on $A$ is the induced subgraph $G[A]$ and (hence by assumption) is not a $\delta$-expander, so we may choose $S$ such that
\[ e_G(S,A \backslash S) < \delta\, \min\{ \vol_H(S), \vol_H(A \backslash S)\} \leq \delta\, \vol_G(S)\,.\]

Let $\hat{G}$ be the final graph, and let $\cQ$ be the final partition $\CC(\hat{G})$ of $V$. Note that each part $Q$ has $\vol_G(Q) \leq \beta \, \vol(G)$, so $q_{\cQ}^D(G) \leq f(\beta)$.
Let $X$ be the set of all edges deleted in the process, so $X = \bigcup_{Q \in \cQ}\, e_G(Q,V\backslash Q)$ and
$q_{\cQ}^E(G) = 1- |X|/ e(G)$.
We shall upper bound $|X|$.\\

Initially assign weight $w(v)=0$ for each $v \in V(G)$. Each time in the `while loop', when we use $S$ to split $A$, the total number of edges deleted increases by $e_G(S,A \backslash S) <  \delta\, \vol_G(S)$. For each $v \in S$ we increase the weight $w(v)$ by $\delta\, \deg_G(v)$, so $\sum_{v \in V} w(v)$ increases by $\delta\, \vol_G(S)$.
Thus at each stage of the process the number of edges deleted so far is at most $\sum_{v \in V} w(v)$; and so for the final weights $w(v)$ at the end of the process, $|X| \leq \sum_{v \in V} w(v)$.  But when a vertex $v$ in a part $A$ has its weight increased, the new part $A'$ containing $v$ has $\vol_G(A') \leq \frac12 \vol_G(A)$, so this can happen at most $\lceil \log_2 \frac1{\beta} \rceil$ times. Thus $w(v) \leq \lceil \log_2 \frac1{\beta} \rceil\, \delta\, \deg_G(v)$ for each $v \in V$, and so
$|X| \leq  \delta\, \lceil \log_2 \frac1{\beta} \rceil \cdot \vol(G)$.

Finally we have
\[\q(G) \geq q_\cQ(G) =
q_{\cQ}^E(G) - q_{\cQ}^D(G) \geq 1 - |X|/e(G) - f(\beta) \geq 1-f(\beta) - 2 \delta\, \lceil \log_2 \tfrac1{\beta} \rceil \,,\]
as required.\end{proof}

%%%%%%%%%%%%%%%%%%%%%%%%%%%%%%%%%%%%%%

%%%%%%%%%%%%%%%%%%%%%%%%%%%%
\needspace{5\baselineskip}
\section{Open questions}
\label{sec.concl}
A direct consequence of Proposition~\ref{prop.exptoqub}, taking $H=G$ and $\alpha=1$, is the following. (Recall that~$h_G$ and~$\hh_G$ are the edge expansion parameters defined in Section~\ref{subsec.defns}.)
\begin{corollary}\label{cor.C3}
For any graph $G$ we have \begin{equation}\label{eq.corC3} q^*(G) \leq 1- \max\{ \hh_G, 1\} \leq 1- h_G.\end{equation}
\end{corollary}
In Proposition~\ref{prop.exptoqub}(ii) we showed that the first inequality in~\eqref{eq.corC3} is tight : for all $\eps>0$ and $0\leq \hdelta \leq 1$, we can find a graph~$G$ with $|\hh_G-\hdelta|<\eps$ and with $q^*(G)\geq 1-\hh_G-\eps$. The construction is a clique with pendant leaves as depicted in Figure~\ref{fig.completeloopANDhairy}(b).\\

However, it is not clear whether the second bound in~\eqref{eq.corC3} is tight for all values of edge expansion. The open question is to find the right upper bound for modularity in terms of $h_G$. 

\begin{open}
What is the optimal function~$f$ for which is it true that for any graph~$G$ we have $\q(G)\leq 1-f(h_G)$?  
\end{open}

We have a family of examples such that $\q(G) \geq 1-2h_G - \eps$. Hence by these examples and Corollary~\ref{cor.C3}, $x \leq f(x) \leq 2x$.\\

For $\delta=1$ we may take $G$ a single edge, then $h_G=1$ and $\q(G)=0$ and thus the bound $\q(G)\leq1-h_G$ in~\eqref{eq.corC3} is tight for $G=K_2$. Similarly for all $\eps>0$ if we take $G$ a sufficiently long path then $h_G<\eps$ and $\q(G)>1-\eps$ which gives an example where the bound in \eqref{eq.corC3} is tight for~$\delta=0$. For $0 < \delta < 1$ we have no such examples, and the question is open. \\

If Corollary~\ref{cor.C3} is tight, this means that a very regular expander exists, in which all big subsets of vertices have roughly the same edge expansion. Otherwise, it would imply that there is a structural reason why, in any graph, big 
subsets of vertices cannot all have the same edge expansion.

\section*{Acknowledgements}
We are grateful to Prasad Tetali for helpful discussions.

\bibliographystyle{plain}
\bibliography{bibli}
\appendix

\end{document}